\documentclass[12pt,reqno]{amsart}

\usepackage{a4wide}

\numberwithin{equation}{section}
\usepackage{mathrsfs}
\usepackage{amsfonts}
\usepackage{amsmath}
\allowdisplaybreaks[4]
\usepackage{stmaryrd}
\usepackage{amssymb}
\usepackage{amsthm}
\usepackage{mathrsfs}
\usepackage{url}
\usepackage{amsfonts}
\usepackage{amscd}
\usepackage{indentfirst}
\usepackage{enumerate}
\usepackage{amsmath,amsfonts,amssymb,amsthm}
\usepackage{amsmath,amssymb,amsthm,amscd}
\usepackage{graphicx,mathrsfs}
\usepackage{appendix}
\usepackage{color}

\usepackage[numbers,sort&compress]{natbib}

\renewcommand{\d}{\mathrm{d}}

\newcommand{\R}{\mathbb{R}}

\newcommand{\Rn}{\mathbb{R}^n}

\renewcommand{\S}{\mathbb{S}}
\numberwithin{equation}{section}

\newtheorem{thm}{Theorem}[section]
\newtheorem{lem}{Lemma}[section]
\newtheorem{prop}{Proposition}[section]

\newtheorem{cor}{Corollary}[section]
\newtheorem{rem}{Remark}[section]

\newcommand{\lam}{\lambda}
\newcommand{\pa}{\partial}
\newcommand{\var}{\varepsilon}
\newcommand{\be}{\begin{equation}}
	\newcommand{\ee}{\end{equation}}

\newcommand{\al}{\alpha}

\begin{document}
		
		\title[]
		{\textbf{New type of bubbling solutions to a critical fractional Schr\"{o}dinger equation
				with double potentials }}
		\author{Ting Li}
		\author{Zhongwei Tang $^{\ast}$}
	\thanks{$^{\ast}$Z. Tang is supported by National Science Foundation of China (12071036, 12126306).}
			\author{Heming Wang}
			\author{Xiaojing Zhang}
		\date{}
		
		\maketitle

		\begin{abstract}
			In this paper, we study the following  critical fractional Schr\"odinger equation:
			\be\label{0.1}
			(-\Delta)^s u+V(|y'|,y'')u=K(|y'|,y'')u^{\frac{n+2s}{n-2s}},\quad u>0,\quad y =(y',y'') \in \R^3\times\R^{n-3},
			\ee	where $n\geq 3$, $s\in(0,1)$,  $V(|y'|,y'')$ and $K(|y'|,y'')$ are two bounded nonnegative potential functions. Under the conditions that $K(r,y'')$ has a stable critical point $(r_0,y_0'')$ with $r_0>0$, $K(r_0,y_0'')>0$ and $V(r_0,y_0'')>0$, we prove that equation \eqref{0.1} has a new type of infinitely many solutions that concentrate at points lying on the top and the bottom of a cylinder. In particular, the bubble solutions can concentrate at a pair of symmetric points with respect to the origin. Our proofs make use of a modified finite-dimensional
			reduction method and local Pohozaev identities. 
		\end{abstract}
		
		

		\section{Introduction and the main results}\label{sec1}
		\setcounter{equation}{0}
		
		In this paper, we consider the following  fractional Schr\"odinger equation with critical exponent and double potentials:
		\be \label{1.1}
		(-\Delta)^s u+V(y)u=K(y)u^{2_s^*-1},\quad u>0 \quad \text { in }\, \Rn,
		\ee
		where $ n \geq 2 s+1$, $0<s<1$, $2_s^*=\frac{2n}{n-2s}$ is the fractional
		critical Sobolev exponent and $V(y)$ and $K(y)$ are bounded nonnegative functions. For any $s \in(0,1)$,  $(-\Delta)^s$ is the fractional Laplacian, which is a nonlocal operator defined as
		$$
		(-\Delta)^s u(y) =c(n, s) \,\mathrm{P.V.}\int_{\Rn} \frac{u(y)-u(x)}{|x-y|^{n+2 s}} \,\d x =c(n, s) \lim _{\var \rightarrow 0^{+}} \int_{\Rn \backslash B_\var(y)} \frac{u(y)-u(x)}{|x-y|^{n+2 s}} \,\d x,
		$$
		where  P.V. is the Cauchy principal value and $c(n, s)=\pi^{-(2 s+\frac{n}{2})} \frac{\Gamma(\frac{n}{2}+s)}{\Gamma(-s)}$. This operator is well defined in $C_{loc}^{1,1}(\Rn) \cap \mathcal{L}_s(\Rn)$, where
		$$
		\mathcal{L}_s(\Rn)=\Big\{u \in L_{l o c}^1(\Rn): \int_{\Rn} \frac{|u(y)|}{1+|y|^{n+2 s}} \, \d y <\infty\Big\} .
		$$ For more details on the fractional Laplace operator, we refer to \cite{DPV2012,CLM2020} and the references therein.
		
		The fractional Laplacian operator emerges in various fields including physics, mathematical finance, and biological modeling. It can be viewed as the infinitesimal generators of stable Lévy processes (see \cite{D2009}). From a mathematical perspective, a significant characteristic of the fractional Laplacian is its non-locality, which makes it more challenging than the classical Laplacian. Therefore, the fractional Laplacian has been extensively studied in both fundamental mathematical research and applications to specific problems, such as \cite{A2017,A20172,BCD2012,BCDS2013,CS2007,DDW2014,JLX2014,YYY2015} and the references therein.
		
		Solutions of problem \eqref{1.1} are related to the existence of standing wave solutions to the following fractional Schr\"odinger equation:
		\begin{equation*}
			\left\{
			\begin{aligned}
				&i\pa_t\Psi+ (-\Delta)^s\Psi=F(x,\Psi) \quad &&\text{ in } \Rn,\\
				&\lim_{|x|\to\infty}|\Psi(x,t)|=0 \quad &&\text{ for all } t>0.
			\end{aligned}
			\right.\end{equation*}
		That is, solutions with the form $\Psi(x,t)=e^{-ict}u(x)$, where $c$ is a constant.

		When $s=1$ and $K(y)\equiv 1$, Chen, Wei and Yan \cite{CWY2012} considered the following nonlinear elliptic equation
		\be\label{0.A}
		-\Delta u+V(|y'|,y'')u=u^{\frac{n+2}{n-2}},\quad u>0,\quad u\in H^1(\Rn).
		\ee
		They proved that \eqref{0.A} has infinitely many non-radial solutions if $n \geq 5$, $V (y)$ is radially symmetric and $r^2V(r)$ has a local maximum point, or a local minimum point $r_0 > 0$ with $V(r_0)>0$. Later, Peng, Wang and Wei \cite{peng2019constructing} constructed infinitely many solutions on a circle under a weak symmetric condition of $V(y)$, where they only required that $r^2V(r,y'')$ has a stable critical point $(r_0,y''_0)$ with $r_0>0$ and $V(r_0,y''_0)>0$. Morover, He, Wang and Wang\cite{HWW2022} proved that the solutions obtained in \cite{peng2019constructing} are nondegenerate. Recently, Duan, Musso and Wei\cite{duan2023doubling} constructed infinitely many solutions, where the bubbles were concentrated at points lying on the top and the bottom circles of a cylinder. Soon after, the above result was generalized to the fractional and polyharmonic cases, see \cite{DHW2023A,GGH2023}. 

		The purpose of the present paper is concerned with the existence of different types of  solutions for \eqref{1.1} with a more complex concentration structure, which cannot be reduced to a two-dimensional
		one. Solutions with similar structure to the related problem can be found in \cite{DHWW2023,GG,CW}. Moreover, due to the presence of the double potentials $V$ and $K$, it is imperative to thoroughly analyze their implications on our problem and identify the more significant roles they play. In contrast to \cite{hww2022}, in our case, the sign of $\Delta K(r_0, y''_0 )$ held no influence, it is $V (r_0, y''_0)$ that assumes a pivotal role in the process of constructing solutions.
		Motivated by these works, we assume that $n\geq 3$, $V(y)$ and $K(y)$ satisfy the following conditions:
		\begin{itemize} 
			\item[$(K_1)$]  $V(y)=V(|y'|, y'')=V(r, y'')$ and $K(y)=K(|y'|, y'')=K(r, y'')$ are nonnegative bounded functions, where $y=(y', y'') \in \mathbb{R}^3 \times \mathbb{R}^{n-3}$.
			
			\item[$(K_2)$] $ K(r, y'')$ has a critical point $y_0=(r_0, y_0'')$ satisfying $r_0>0$, $K(r_0, y_0'')>0$ and $$\deg\Big(\nabla(K(r,y''),(r_0,y''_0))\Big)\neq 0.$$
			\item[$(K_3)$] $V(r, y'') \in C^{2}(B_\kappa(r_0, y_0''))$, $K(r, y'') \in C^{3}(B_\kappa(r_0, y_0''))$, where $\kappa>0$ is a small constant and $V(r_0, y''_0) >0$.
			
			Without loss of generality, we assume that
			$$
			K(r_0, y_0'')=1.
			$$
		\end{itemize}
		\begin{rem}
			Recently, Du, Hua and Wang \cite{DHW2023A} constructed a new type of solutions for equation \eqref{1.1} under the condition $K\equiv1$. However, it is clear from the condition $(K_2)$ that our situation is totally different.
		\end{rem}
		
		Let $D$ be the completion of the space $C_c^{\infty}(\Rn)$ under the norm$$
		\|(-\Delta)^{\frac{s}{2}} u\|_{L^2(\Rn)}=\Big(\int_{\Rn} (1+|\eta|^{2 s})|\mathscr{F} u(\eta)|^2 \,\d \eta\Big)^{\frac{1}{2}},
		$$
		where $\mathscr{F} u(\eta)=\frac{1}{(2 \pi)^{\frac{n}{2}}} \int_{\Rn}  e^{-i \eta \cdot x} u(x) \,\d x $ is the Fourier transformation of $u$.
		
		We will construct solutions in the following space
         $$H^s(\Rn)=\Big\{u \in D: \int_{\R^N} V(y)u^2 \d y < \infty\Big\},$$
		with the norm
		$$\|u\|_{H^s(\Rn)}=\Big(\|(-\Delta)^{\frac{s}{2}}u\|_{L^2(\Rn)}^2+\int_{\Rn}V(y)u^2 \,\d y\Big)^\frac{1}{2}.$$
		
		For any $x \in \Rn$ and $\lam  >0$, set
		$$
		U_{x, \lam}(y)=2^{\frac{n-2s}{2}}\Big( \frac{\Gamma(\frac{n+2 s}{2})}{\Gamma(\frac{n-2s}{2})}\Big)^{\frac{n-2s}{4 s}}\Big(\frac{\lam}{1+\lam  ^2|y-x|^2}\Big)^{\frac{n-2s}{2}}.
		$$
		Then for every $(x, \lam  ) \in \Rn\times(0,\infty)$, $U_{x, \lam  }$ is the solution to
		\be\label{1.4}
		(-\Delta)^s u=u^{\frac{n+2 s}{n-2s}}, \quad u>0 \quad \text { in } \,\Rn.
		\ee
		Moreover, the functions
		$$Z_0(y)=\frac{\pa  U_{x,\lam }(y)}{\pa  \lam }\Big|_{\lam =1} , \quad Z_i(y)=\frac{\pa  U_{x,\lam }(y)}{\pa  y_i}, \quad i=1,2,\ldots,n$$
		solving the linearized problem
		$$(-\Delta)^s \varphi =(2_s^*-1)U_{x,\lam }^{2_s^*-2}\varphi,\quad \varphi>0 \quad \text{ in }\,\Rn,$$
		are the kernels of the linearized operator associated with problem \eqref{1.4}.
		
		Define the symmetric Sobolev space
		\begin{equation*}
			\begin{aligned}
				H_s=\Big\{u : &u \in H^s(\Rn), u \text{ is even in}~ y_2, y_3, \\ &u(r\cos\theta,r\sin\theta,y_3,y'')=u\Big(r\cos\big(\theta+\frac{2j\pi}{k}\big),r\sin\big(\theta+\frac{2j\pi}{k}\big),y_3,y''\Big),j=1,\ldots,k\Big\},
			\end{aligned}
		\end{equation*}
		where $r=\sqrt{y_1^2+y_2^2}$ and $\theta=\arctan\frac{y_2}{y_1}$.
		
		Let 
		$$x_j^{\pm}=\Big(\bar{r} \sqrt{1-\bar{h} ^2} \cos \frac{2(j-1)\pi}{k}, \bar{r} \sqrt{1-\bar{h}^2}\sin\frac{2(j-1)\pi}{k},\pm\bar{r}\bar{h},\bar{y}''\Big),  \quad j=1,\ldots,k,$$
		where $\bar{y}''$ is a vector in $\R^{n-3}$, $\bar{h}\in (0,1)$ and $(\bar{r},\bar{y}'')$ is close to $(r_0,y_0'')$.
		
		Compared with the construction of bubble solutions, we need to overcome the difficulty caused by the parameter $\bar{h}$. In fact, expansions for the main term of energy functional are determined by the order of $\bar{h}$. Thus we should consider the order of $\bar{h}$. In this paper, we consider three cases of $\bar{h}$ in the process of constructing solutions:
		(i) $\bar{h}$ goes to 1;
		(ii) $\bar{h}$ is separated from 0 and 1;
		(iii) $\bar{h}$ goes to 0.
		
		We would like to point out that in case (ii) the distance between these points $\{x^+_j\}_{j=1}^k$ ($\{x^-_j\}_{j=1}^k$) is $2r_0\bar{h}$
		which is quite similar to the corresponding points in \cite{peng2019constructing}. While in case (i), $\{x^+_j\}_{j=1}^k$ and $\{x^-_j\}_{j=1}^k$ will go to the North pole and the South pole of $\S^2:=\{(x_1,\ldots,x_n)\in \Rn: x_1^2+x_2^2+x_3^2=r_0^2, (x_4,\ldots,x_n)=y_0''\}$ simultaneously. In case (iii), $\{x^+_j\}_{j=1}^k$ and $\{x^-_j\}_{j=1}^k$ will go to the circle $\S_0^1:=\{(x_1,\ldots,x_n)\in \Rn: x_1^2+x_2^2=r_0^2, x_3=0, (x_4,\ldots,x_n)=y_0''\}$ at the same time, so the distance between these points $\{x^+_j\}_{j=1}^k$, $\{x^-_j\}_{j=1}^k$ may be very close.
		
		The idea of constructing solutions is to glue some $U_{x_j^\pm,\lam }$ as an approximation solution. Similar to \cite{peng2019constructing}, in order to not only deal with the slow decay of this function when $n$ is not big, but also simplify some computations, we introduce the smooth cut off function $\eta(y)=\eta(|y'|,y'')$ satisfying $\eta=1$ if $|(r,y'')-(r_0,y_0'')|\leq\sigma$, $\eta=0$ if $|(r,y'')-(r_0,y_0'')|\geq 2\sigma$, $0\leq\eta\leq1$ and $|\nabla \eta|\leq C$, where $\sigma>0$ is a small constant such that $V(r,y'')>0$, $K(r,y'')>0$ if $|(r,y'')-(r_0,y_0'')|\leq10\sigma$.
		
		Denote
		$$Z_{x_j^{\pm},\lam }=\eta U_{x_j^{\pm},\lam },\quad Z_{\bar{r},\bar{h},\bar{y}'',\lam }^*=\sum_{j=1}^{k} U_{x_j^{+},\lam }+\sum_{j=1}^{k} U_{x_j^{-},\lam },\quad Z_{\bar{r},\bar{h},\bar{y}'',\lam }=\sum_{j=1}^{k} \eta U_{x_j^{+},\lam }+\sum_{j=1}^{k} \eta U_{x_j^{-},\lam }.$$

		As for the case (i), we assume that $\al =1-\nu$, $k>0$ is a large integer, $\lam  \in [L_0k^{\frac{n-2s}{n-4s-\al }},L_1k^{\frac{n-2s}{n-4s-\al }}]$ for some constants $L_1>L_0>0$ and $(\bar{r},\bar{h},\bar{y}'')$ satisfies
		\be \label{1.5}
		|(\bar{r},\bar{y}'')-(r_0,y''_0)|\leq \theta_k, \quad \sqrt{1-\bar{h}^2}=\frac{1}{M_1\lam ^{\frac{\al }{n-2s}}},
		\ee
		where $\theta_k=k^{\frac{(2s-n)(2s+1-\beta)}{2(n-4s-\al)}}$ with $\beta=\frac{\al}{n-4s}$, $\nu>0$  is a small constant and $M_1$ is a positive constant.
		
		For $n\geq4$, we assume that $s$ satisifies:
		\be
		\left\{\begin{aligned}\label{0}
			&\frac{5+2n-\sqrt{4n^2-2n+25}}{8}<s<\frac{3n-3-\sqrt{n^2-2n+9}}{8},&&\quad \text{ if } n=4,5,\\
			&\frac{3n-3+\sqrt{n^2-2n+9}}{8}<s<1,&&\quad \text{ if } n\geq6.
		\end{aligned}\right.
		\ee
		\begin{thm}\label{theo1.1}
			Suppose that $n\geq 4$ and $s$ satisfies \eqref{0}. If $V(y)$ and $K(y)$ satisfies $(K_1)$-$(K_3)$, then there exists a positive integer $k_0>0$, such that for any integer $k>k_0$, problem $(\ref{1.1})$ has a solution $u_k$ of the form
			\be \label{1.6}
			u_k=Z_{\bar{r},\bar{h},\bar{y}'',\lam _k}+\varphi_k,
			\ee
			where $\lam _k \in [L_0k^{\frac{n-2s}{n-4s-\al }},L_1k^{\frac{n-2s}{n-4s-\al }}]$ and $\varphi_k \in H_s$. Moreover, as $k \to \infty$, $|(\bar{r}_k,\bar{y}''_k)-(r_0,y''_0)|= o(1)$,  $\sqrt{1-\bar{h}_k^2}=\frac{1}{M_1\lam _k^{\frac{\al }{n-2s}}}$ and $\lam _k^{-\frac{n-2s}{2}}\|\varphi_k\|_{L^{\infty}(\Rn)} \to 0$.
		\end{thm}
		\begin{rem}
			It is remarkable that when $n\geq4$, the hypothesis  \eqref{0} is pivotal in Lemma \ref{lem2.4},  as it ensures the existence of a small positive constant $\nu$. Actucally, it is equivalent to $\frac{n-4s}{n-2s}<2s-\frac{1}{2}$, $\frac{n-4s}{n-2s}<\frac{n-4s}{2}-\frac{s}{n-2s}$ and $n>4s+1$.
		\end{rem}
		\begin{rem}\label{rema1.3}
			Note that the bubble solutions $u_k$ concentrate at two kinds of points $\{x_{j}^{\pm}\}^{k}_{j=1}$ which are axisymmetric on the third coordinate and the number of bubbles can be made arbitrarily large. In particular, when taking $y_0''=0\in \R^{n-3}$, $u_k$ concentrates on a pair of symmetric points with respect to the region.
		\end{rem}
		
		Additionally, as for the case (ii) and case (iii), we assume that $k>0$ is a large integer, $\lam  \in [L'_0k^{\frac{n-2s}{n-4s}},L'_1k^{\frac{n-2s}{n-4s}}]$ for some constants $L'_1>L'_0>0$ and $(\bar{r},\bar{h},\bar{y}'')$ satisfies
		\be \label{1.7}
		|(\bar{r},\bar{y}'')-(r_0,y''_0)|\leq \bar{\theta}_k, \quad \bar{h}=a+\frac{1}{M_2\lam ^{\frac{n-4s}{n-2s}}},
		\ee
		where ${\theta}_k=k^{\frac{(2s-n)(2s+1-\beta')}{2(n-4s)}}$ with $\beta'=\frac{n-2s}{n-4s}$, $a\in[0,1)$ is a constant and $M_2$ is a positive constant.

		\begin{thm}\label{theo1.4}
			Suppose that $n\geq 4$ and $s$ satisfies \eqref{0}. If $V(y)$ and $K(y)$ satisfies $(K_1)$-$(K_3)$, then there exists a positive integer $k_0>0$, such that for any integer $k>k_0$, problem \eqref{1.1} has a solution $u_k$ of the form
			\begin{equation*}
				u_k=Z_{\bar{r},\bar{h},\bar{y}'',\lam _k}+\varphi_k,
			\end{equation*}
			where $\lam _k \in [L'_0k^{\frac{n-2s}{n-4s}},L'_1k^{\frac{n-2s}{n-4s}}]$ and $\varphi_k\in H_s$. Moreover, as $k \to \infty$, $|(\bar{r}_k,\bar{y}''_k)-(r_0,y''_0)|= o(1)$,  $\bar{h}_k=a+\frac{1}{M_2\lam_k ^{\frac{n-4s}{n-2s}}}$ and $\lam _k^{-\frac{n-2s}{2}}\|\varphi_k\|_{L^{\infty}(\Rn)} \to 0$.
		\end{thm}
		
		\begin{rem}\label{rema1.4}
			We would like to point out that since there involves $\bar{h}$ in the points $\{x_{j}^{\pm}\}_{j=1}^{k}$ which yield that $\frac{\pa }{\pa  \lam }U_{x_{j}^{\pm},\lam }=O\big(\frac{1}{\lam ^{\beta}}U_{x_{j}^{\pm},\lam })$ not $\frac{\pa }{\pa  \lam }U_{x_{j}^{\pm},\lam }=O\big(\frac{1}{\lam }U_{x_{j}^{\pm},\lam })$, where $\beta$ equals to $\frac{\al }{n-2s}$ and $\frac{n-4s}{n-2s}$ in Theorem \ref{theo1.1} and Theorem \ref{theo1.4}, respectively,
			to obtain the good enough estimates for the high order terms, we have to deal with all the estimates more carefully, and we have to improve the accuracy of the estimates for the error term $\varphi$ (see Proposition \ref{prop2.5}).
		\end{rem}
		
		\begin{rem}\label{rema1.5}
			We would like to point out that the solutions we obtained are different from those obtained in \cite{guo2020solutions,guo2013infinitely}.
			Meanwhile, they are also different from the ones obtained in \cite{duan2023doubling} where they only consider the case that $\bar{h}$ goes to $0$.
		\end{rem}
		Similar to \cite{DHWW2023}, we can construct $2k$-pieces of infinitely many-bubble
		solutions symmetric with respect to the third coordinate of problem \eqref{1.1}.
		\begin{cor}\label{cor1.1}
			Under the assumptions stated in Theorem \ref{theo1.4}, for any fixed integer $l \in \mathbb{N}^{+}$, there exists an integer $k_0>0$, such that for all integers $k \geq k_0$, the equation \eqref{1.1} has infinitely many solutions symmetric with respect to the $y_3$-coordinate, which can be expressed as
			$$
			u_{k, l}=\sum_{i=1}^l \sum_{j=1}^k \eta U_{x_{i, j}^{+}, \lam  }+\sum_{i=1}^l \sum_{j=1}^k \eta U_{x_{i, j}^{-}, \lam  }+\varphi_{k, l},
			$$
			\noindent where $\varphi_{k, l} \in H_s$ and for $j=1, \ldots, k$,
			$$
			x_{j, i}^{ \pm}=\Big(\bar{r} \sqrt{1-\bar{h}_j^2} \cos \frac{2(i-1) \pi}{k}, \bar{r} \sqrt{1-\bar{h}_j^2} \sin \frac{2(i-1) \pi}{k}, \pm \bar{r} \bar{h}_j, \bar{y}''\Big), \quad i=1, \ldots, l,
			$$
			with $\bar{h}_j-\bar{h}_{j-1} \in(\frac{1}{M_2 \lam  ^{\frac{n-4 s}{n-2s}}}, a+\vartheta'')$, where $\vartheta''>0$ is a small constant, $\bar{h}_0=0$ and $(\bar{r}, \bar{y}'')$ is close to $(r_0, y_0'')$.
		\end{cor}
		Now, we outline the main ideas in the proofs of Theorems \ref{theo1.1} and \ref{theo1.4}, and discuss the primary difficulties encountered in proving the desired results.
		
		The proofs of Theorems \ref{theo1.1} and \ref{theo1.4} are based on a modified finite-dimensional Lyapunov-Schmidt reduction and local
		Pohozaev-type identities. The finite-dimensional reduction method has been extensively used to construct solutions for equations with critical growth, we refer to \cite{BC1991,CNY2002,guo2013infinitely,guo2016infinitely,GLN2019,L1996,wei2010infinitely,Y2000,SZ1996} and the references therein. We carry out the reduction procedure in a space with weighted maximum norm, similar weighted maximum norm has been used in \cite{del2003two,wei2005arbitrary,guo2020solutions,peng2019constructing,PWY2018,wei2010infinitely}. We initially define an approximation and seek a solution within a nearby neighborhood. This is achieved by linearizing around the approximation and, subsequent to developing an appropriate linear theory, solving through a fixed point argument. This allows to reduce the original problem to the solvability of a finite dimensional one. Then we turn to find suitable $(\bar{r},\bar{h}, \bar{y}'',\lam  )$. However, for the construction in Theorem \ref{theo1.1} (or Theorem \ref{theo1.4}), this final step is quite delicate. In fact, We know that the functional corresponding to equation \eqref{1.1} is
		\be \label{1.10}
		I(u)=\frac{1}{2} \int_{\Rn} \big(|(-\Delta)^{\frac{s}{2}} u|^2+V(y)u^2 \big)\,\d y-\frac{1}{2_s^*} \int_{\Rn}  K(y)(u)_+^{2_s^*}\,\d y .
		\ee
		By the reduction argument, the problem of finding a critical point for $I(u)$ with the form \eqref{1.6} can be reduced to that of finding a critical point of the following perturbed function
		\begin{equation*}
			F(\bar{r},\bar{h},\bar{y}'',\lam )=I(Z_{\bar{r},\bar{h},\bar{y}'',\lam }+\varphi_{\bar{r},\bar{h},\bar{y}'',\lam }),
		\end{equation*}
		where $\bar{r},\bar{h},\bar{y}''$ and $\lam $ satisfy those conditions in Theorem \ref{theo1.1} or Theorem \ref{theo1.4}, respectively. To find the critical point, it is essential to obtain a good estimate for the error term. As in\cite{DHWW2023}, we  will carry out the reduction argument in a weighted space and obtain the following estimate for the error term $\varphi$
		$$	\|\varphi\|_{*} =O\Big( \frac{1}{\lam ^{\frac{2s+1-\beta}{2}+\var }}\Big),$$
		where $\var>0$ is a small sonstant, $\beta$ equals to $\frac{\al }{n-2s}$ and $\frac{n-4s}{n-2s}$ in Theorem \ref{theo1.1} and Theorem \ref{theo1.4}, respectively. Using the facts $\frac{\pa }{\pa  \lam }U_{x_{j}^{\pm},\lam }=O\big(\frac{1}{\lam ^{\beta}}U_{x_{j}^{\pm},\lam })$, $\frac{\pa }{\pa  y_i }U_{x_{j}^{\pm},\lam }=O\big(\lam   U_{x_{j}^{\pm},\lam })$, we have the following estimates (see Lemma \ref{lemB.2}, \eqref{r} and \eqref{y})
		\be\label{1.11}
		\frac{\pa F}{\pa\lam  }=\!2k\Big(\!-\!\frac{sB_1}{\lam ^{2s+1}}\!+\!\sum_{j=2}^{k}\frac{B_2}{\lam ^{n-2s+1}|x_j^+-x_1^+|^{n-2s}}
		\!+\!\sum_{j=1}^{k}\frac{B_2}{\lam ^{n-2s+1}|x_j^--x_1^+|^{n-2s}}\!+\!O\Big(\frac{1}{\lam ^{2s+1+\var }}\Big)\Big),\ee
		\be\label{1.12}
		\frac{\pa F}{\pa\bar{r}} =2k\Big(\frac{sB'_1}{\lam ^{2s}}+\sum_{j=2}^{k}\frac{B'_2}{\bar{r}\lam ^{n-2s}|x_j^+-x_1^+|^{n-2s}}
		+\sum_{j=1}^{k}\frac{B'_2}{\bar{r}\lam ^{n-2s}|x_j^--x_1^+|^{n-2s}}+O\Big(\frac{1}{\lam ^{s+\var }}\Big)\Big),
		\ee
		\be\label{1.13}\frac{\pa F}{\pa\bar{y}''_j}=2k\Big(\frac{sB'_1}{\lam ^{2s}}+O\Big(\frac{1}{\lam ^{s+\var }}\Big)\Big),
		\ee
		where $B_i$ and $B_i^{'}$ are some positive constants. From the above estimates, we deduce that the error term $\varphi$ does not destroy the main terms in the expansion of \eqref{1.11}. But it destroys the main terms in the expansions of \eqref{1.12} and \eqref{1.13}.
		This makes it impossible to identify a critical point for $F(\bar{r},\bar{h},\bar{y}'',\lam )$. To overcome this difficulty, motivated by \cite{PWY2018}, we will choose appropriate $(\bar{r},\bar{h}, \bar{y}'')$ in a suitable neighborhood $B_\rho$ of $(r_0, y_0'')$  satisfying the following local Pohozaev identities:
		\begin{align*}
			&-\int_{\pa''\mathcal{B}_\rho^+} t^{1-2s}\frac{\pa\tilde{u}_k}{\pa \nu}\frac{\pa\tilde{u}_k}{\pa y_i}+\frac{1}{2}\int_{\pa''\mathcal{B}_\rho^+} t^{1-2s}|\nabla\tilde{u}_k|^2\nu_i
			\\=&\int_{B_\rho}\Big(-V(r,y'')u_k+K(r,y'')(u_k)_+^{2_s^*-1}\Big)\frac{\pa u_k}{\pa y_i},\quad i=4,\ldots,n,
		\end{align*}
		and
		\begin{align*}	
			&-\int_{\pa''\mathcal{B}_\rho^+} t^{1-2s}\langle\nabla\tilde{u}_k,Y\rangle\frac{\pa\tilde{u}_k}{\pa \nu}+\frac{1}{2}\int_{\pa''\mathcal{B}_\rho^+}t^{1-2s}|\nabla\tilde{u}_k|^2\langle Y,\nu\rangle +\frac{2s-n}{2}\int_{\pa\mathcal{B}_\rho^+}t^{1-2s}\frac{\pa\tilde{u}_k}{\pa \nu}\tilde{u}_k\\
			=&\int_{{B_\rho}}\Big(-V(r,y'')u_k+K(r,y'')(u_k)_+^{2_s^*-1}\Big)\langle \nabla u_k,y\rangle,
		\end{align*}
		where $u=Z_{\bar{r},\bar{h}, \bar{y}'', \lam }+\varphi_{\bar{r}, \bar{y}'', \lam }$ is the function obtained by the reduction argument, $\nu$ is the unit outer normal of $\pa\mathcal{B}_\rho^+$, $\nu_i$ is the $i$-th component of $\nu$, $\tilde{u}$ is the extension of $u$ (see  \eqref{3.A} below) and
		\begin{equation*}
			\begin{aligned}
				B_\rho&=\{y\in\Rn:|y-(r_0,y_0'')|\leq\rho\},\\
				\mathcal{B}_\rho^+&=\{Y=(y,t)\subset \R_+^{n+1}:|(y,t)-(r_0,y_0'',0)|\leq\rho,t>0\},\\
				\partial'\mathcal{B}_\rho^+&=\{Y=(y,t)\subset \Rn\times\{0\}:|y-(r_0,y_0'')|\leq\rho,t=0\},\\
				\partial''\mathcal{B}_\rho^+&=\{Y=(y,t)\subset \R_+^{n+1}:|(y,t)-(r_0,y_0'',0)|=\rho,t>0\}.
			\end{aligned}
		\end{equation*}
		Then $\partial \mathcal{B}_\rho^+=\partial'\mathcal{B}_\rho^+\cup\partial''\mathcal{B}_\rho^+$.
		
		For any $u \in D$, $\tilde{u}$ is defined by
		\begin{equation}\label{3.A}
			\tilde{u}(y,t)=\mathcal{P}_s[u]=\int_{\Rn}P_s(y-\xi,t)u(\xi)\,\d\xi, \quad (y,t)\in \R_+^{n+1}=\Rn\times(0,\infty),
		\end{equation}
		where
		\begin{equation*}
			P_s(x,t)=\beta(n,s)\frac{t^{2s}}{(|x|^2+t^2)^{\frac{n+2s}{2}}}
		\end{equation*}
		with the constant $\beta(n,s)$ satisfying $\int_{\Rn}P_s(x,1)\,\d x=1$. Thus $\tilde{u} \in L^{2}(t^{1-2s},\Omega)$ for any compact set $\Omega$ in $\overline{\R_+^{n+1}}$, $\nabla\tilde{u} \in L^2(t^{1-2s},\R_+^{n+1})$ and $\tilde{u} \in C^\infty(\R_+^{n+1})$. Moreover, $\tilde{u}$ satisfies (see \cite{CS2007})
		\begin{equation*}
			\left\{
			\begin{aligned}
				&\text{ div }(t^{1-2s}\nabla\tilde{u})=0 &&\quad \text{ in } \R_+^{n+1},\\
				&\tilde{u}(y,0)=u(y)&&\quad \text{ on }\Rn,\\
				&-\lim\limits_{t \to 0}t^{1-2s}\pa_t\tilde{u}(y,t)=\omega_s(-\Delta)^s u(y) &&\quad \text{ on }\Rn,\\
			\end{aligned}	\right.
		\end{equation*}
		where $\omega_s=2^{1-2s}\Gamma(1-s)/\Gamma(s)$. Then all $c_l$ $(l=2,\ldots,n)$ in \eqref{2.21} are equal to zero, which implies that $u=Z_{\bar{r},\bar{h}, \bar{y}'', \lam }+\varphi_{\bar{r}, \bar{y}'', \lam }$ is the solution of equation \eqref{1.1}.
		
		We would like to emphasize that we obtain $(\bar{r},\bar{h},\bar{y}'')$ through the reduction argument, instead of determining these parameters directly by computing the derivatives of the reduced function $F(\bar{r},\bar{h},\bar{y}'',\lam )$ with respect to $\bar{r},\bar{h}$, and $\bar{y}_k''$ for $k=4,\ldots,n$. In fact, we cannot impose the condition $\frac{\partial F}{\partial \bar{h}}=0$ or the equivalent Pohozaev identity, as it leads to $\bar{h}$ approaching 0, which contradicts case (i) and case (ii) when $k$ approaches infinity. To overcome this difficulty, we necessitate that the bubbling solutions are symmetric about the third coordinate axis. Subsequently, we can alleviate the restriction on the decay rate of $\bar{h}$ by employing a modified reduction method.

		The primary objective of this paper is to construct new types of bubble solutions analogous to those presented in \cite{duan2023doubling}, where they constructed $2k$-bubble solutions that concentrate at points lying on the upper and lower circles of a cylinder, under the assumption that $K(y)$ is a radially symmetric function. To release the assumptions on the function $K(y)$, inspired by \cite{PWY2018,guo2020solutions}, we attempt to construct solutions by utilizing various local Pohozaev identities to derive algebraic equations that determine the location of the bubbles. Therefore, we can construct $2k$-bubble solutions that are symmetric with respect to the third coordinate and concentrate at even the saddle points of $K(y)$, which possess a weaker sense of symmetry. However, in the present paper, due to the possibility that the bubbles of our solutions may be extremely close when $\bar{h}$ approaches 0 or 1, we are required to perform more intricate computations and obtain more precise estimates. We will discuss this in more details later. On the other hand, owing to the apperance of double potential, we must analyze their roles in our problem. Indeed, the condition $V(r_0, y''_0)>0$ guarantees the main terms in the expansion of the energy function as stated in Lemma \ref{lemB.1}, while the conditions of $K$ are employed for the ultimate proofs of Theorems \ref{theo1.1} and  \ref{theo1.4}.
		
		The present paper is organized as the following. In  Section \ref{sec:2}, we carry out a finite-dimensional
		reduction procedure. Then the expansions of the energy for the approximate solutions are proved
		in Section \ref{sec:3}. Finally, we prove the main results in Section \ref{sec:4} by using local Pohozaev identities. We put some basic estimates in the
		Appendix \ref{secA}. 
		
		{\bf Notations:} We collect below a list of the main notations used throughout the paper.
		\begin{itemize}
			\item The notation $o(1)$ represents some quantity that converges to zero as $k$ approaches infinity.
			\item We will write that $a \lesssim b$ (resp. $a \gtrsim b$ ) if $a \leq Cb$ (resp. $Ca \geq b$), where $C$ is a constant
			depends only on $n$, $s$ and others but independent on the number of bubbles $k$.
			\item The integral $\int$ always means $\int_{\Rn}$ unless specified.
			\item The standard scalar product in $\Rn$ is denoted by $\langle\cdot,\cdot\rangle_{\Rn}$.
		\end{itemize}
		
		\section{The Finite-Dimensional Reduction}\label{sec:2}
		
		In this section, we perform a finite-dimensional reduction by using $Z_{\bar{r},\bar{h},\bar{y}'',\lam }$ as an approximation solution. For later calculations, we divide $\Rn$ into $k$ parts, for $j=1,\ldots,k$,
		\begin{align}
			\Omega_j =\Big\{&y=(y_1,y_2,y_3,y'')\in \R^3\times\R^{n-3}: \notag\\
			&\Big\langle \frac{(y_1,y_2)}{|(y_1,y_2)|},\Big(\cos\frac{2(j-1)\pi}{k},\sin\frac{2(j-1)\pi}{k}\Big)\Big\rangle_{\R^2}\geq\cos\frac{\pi}{k}\Big\},\label{2.1}
		\end{align}
		where $\langle\cdot,\cdot\rangle_{\R^2}$ denote the dot product in $\R^2$. For $\Omega_j$, we further divide it into two separate parts:
		\be\label{2.2}\Omega_j^+=\big\{y:y=(y_1,y_2,y_3,y'') \in \Omega_j, y_3 \geq 0\big\},\ee
		$$\Omega_j^-=\big\{y:y=(y_1,y_2,y_3,y'') \in \Omega_j, y_3 < 0\big\}.$$
		It is easy to verify that
		\begin{equation*}
			\Rn=\bigcup\limits_{j=1}^{k}\Omega_j,\quad \Omega_j=\Omega_j^+\cup\Omega_j^-
		\end{equation*}
		and
		\begin{equation*}
			\Omega_j \cap \Omega_i=\emptyset, \quad \Omega^+_j \cap \Omega^-_j=\emptyset, \quad \text{if}~i\neq j.
		\end{equation*}
		
		For $u,f\in L^\infty(\Rn)$, let
		$$\|u\|_*=\sup_{y\in\Rn}\Big(\sum_{j=1}^{k}\Big(\frac{1}{(1+\lam |y-x_j^+|)^{\frac{n-2s}{2}+\tau}}+\frac{1}{(1+\lam |y-x_j^-|)^{\frac{n-2s}{2}+\tau}}\Big)\Big)^{-1}\lam ^{-\frac{n-2s}{2}}|u(y)|$$
		and
		$$\|f\|_{**}=\sup_{y\in\Rn}\Big(\sum_{j=1}^{k}\Big(\frac{1}{(1+\lam |y-x_j^+|)^{\frac{n+2s}{2}+\tau}}+\frac{1}{(1+\lam |y-x_j^-|)^{\frac{n+2s}{2}+\tau}}\Big)\Big)^{-1}\lam ^{-\frac{n+2s}{2}}|f(y)|,$$
		where $\tau=\frac{n-4s-\al }{2(n-2s-\al)}<\frac{n-4s}{2(n-2s)}$. 
		
		Note that the maximum norm will not be affected by the number of the bubbles, we need to carry out the reduction procedure in a space with weighted maximum norm, similar weighted maximum norm has been used in \cite{del2003two,wei2005arbitrary,guo2020solutions,peng2019constructing,PWY2018,wei2010infinitely}.

		For $j=1,2,\ldots,k$, denote
		$$Z_{j,2}^\pm=\frac{\pa  Z_{x_j^\pm,\lam }}{\pa \lam },\quad Z_{j,3}^\pm=\frac{\pa  Z_{x_j^\pm,\lam }}{\pa  \bar{r}},\quad Z_{j,l}^\pm=\frac{\pa  Z_{x_j^\pm,\lam }}{\pa  \bar{y}''_l},\quad l=4,\ldots,n.$$
		
		We also define the constrained space
		\begin{align*}
			\mathbb{H}=\Big\{v:v\in H_s,\int Z_{x_j^+,\lam }^{2_s^*-2}Z_{j,l}^+ v=\int Z_{x_j^-,\lam }^{2_s^*-2}Z_{j,l}^- v=0,~j=1,\ldots,k,~l=2,\ldots,n\Big\}.
		\end{align*}
		
		Now, we consider the following linearized problem in $ \mathbb{H}$: 
		\begin{align}
			&(-\Delta)^s\varphi+V(r,y'')\varphi-(2_s^*-1)K(r,y'')Z_{\bar{r},\bar{h},\bar{y}'',\lam }^{2_s^*-2}\varphi\notag\\=&f+\sum_{l=2}^n
			c_l\sum_{j=1}^k\Big(Z_{x_j^+,\lam }^{2_s^*-2}Z_{j,l}^+ +Z_{x_j^-,\lam }^{2_s^*-2}Z_{j,l}^-\Big)\quad\text{ in }\ \Rn,\label{a.2.1} 
		\end{align}
		where $c_l $ $(l=2,\ldots,n)$ are some real numbers.
		
		In the sequel of this section, we assume that $(\bar{r},\bar{y}'')$ and $\bar{h}$ satisfy \eqref{1.5}.
		
		\begin{lem}\label{lem2.1}
			Suppose that 
			$\varphi_k$ solves \eqref{a.2.1} for $f=f_k$. If $\|f_k\|_{**}$ goes to zero as $k$ goes to infinity, so does $\|\varphi_k\|_*$.
		\end{lem}
		\begin{proof}
			We argue by contradiction. Suppose that there exist $k \to \infty$, $f=f_k$, $\lam =\lam _k \in [L_0 k^{\frac{n-2s}{n-4s-\al }},$ $L_1k^{\frac{n-2s}{n-4s-\al }}]$, $(\bar{r}_k,\bar{h}_k,\bar{y}''_k)$ satisfying \eqref{1.5}, $\varphi_k$ solving \eqref{a.2.1} for $f=f_k$, $\lam =\lam _k$, $\bar{r}=\bar{r}_k$, $\bar{h}=\bar{h}_k$, $\bar{y}''=\bar{y}''_k$ with $\|f_k\|_{**}\to 0$ and $\|\varphi_k\|_{*}\geq c> 0$. Without loss of generality, we may assume that $\|\varphi_k\|_{*}=1$. For simplicity, we drop the subscript $k$.
			
			From equation \eqref{a.2.1}, we have
			\begin{align*}	
				|\varphi(y)|\lesssim& \int\frac{K(|z'|,z'')}{|z-y|^{n-2s}}Z_{\bar{r},\bar{h},\bar{y}'',\lam }^{2_s^*-2}|\varphi|\, \d z +\int\frac{1}{|z-y|^{n-2s}}|f|\, \d z \notag\\
				& +\int\frac{1}{|z-y|^{n-2s}}\Big|\sum_{l=2}^n
				c_l\sum_{j=1}^k(Z_{x_j^+,\lam }^{2_s^*-2}Z_{j,l}^+ +Z_{x_j^-,\lam }^{2_s^*-2}Z_{j,l}^-)\Big|\, \d z. 
			\end{align*}
			
			Using Lemma \ref{lemA.3}, we can deduce that
			\begin{align*}
				&\int\frac{K(|z'|,z'')}{|z-y|^{n-2s}}Z_{\bar{r},\bar{h},\bar{y}'',\lam }^{2_s^*-2}|\varphi|\, \d z\notag\\
				\lesssim& \|\varphi\|_*\lam ^{\frac{n-2s}{2}}\int\frac{Z_{\bar{r},\bar{h},\bar{y}'',\lam }^{2_s^*-2}}{|z-y|^{n-2s}}\Big(\sum_{j=1}^{k}\frac{1}{(1+\lam |z-x_j^+|)^{\frac{n-2s}{2}+\tau}}+\sum_{j=1}^{k}\frac{1}{(1+\lam |z-x_j^-|)^{\frac{n-2s}{2}+\tau}}\Big)\, \d z\notag\\
				\lesssim &\|\varphi\|_*\lam ^{\frac{n-2s}{2}}\Big(\sum_{j=1}^{k}\frac{1}{(1+\lam |y-x_j^+|)^{\frac{n-2s}{2}+\tau+\theta}}+\sum_{j=1}^{k}\frac{1}{(1+\lam |y-x_j^-|)^{\frac{n-2s}{2}+\tau+\theta}}\Big),
			\end{align*}
			where $\theta>0$ is a small constant. Similarily, it follows from Lemma \ref{lemA.2} that
			\begin{align*}
				&\int\frac{1}{|z-y|^{n-2s}}|f|\, \d z\notag\\
				\lesssim&\|f\|_{**}\lam ^{\frac{n+2s}{2}}\int\frac{1}{|z-y|^{n-2s}}\Big(\sum_{j=1}^{k}\frac{1}{(1+\lam |z-x_j^+|)^{\frac{n+2s}{2}+\tau}}+\sum_{j=1}^{k}\frac{1}{(1+\lam |z-x_j^-|)^{\frac{n+2s}{2}+\tau}}\Big)\, \d z\notag\\
				\lesssim &\|f\|_{**}\lam ^{\frac{n-2s}{2}}\Big(\sum_{j=1}^{k}\frac{1}{(1+\lam |y-x_j^+|)^{\frac{n-2s}{2}+\tau}}+\sum_{j=1}^{k}\frac{1}{(1+\lam |y-x_j^-|)^{\frac{n-2s}{2}+\tau}}\Big).
			\end{align*}
			From the definitions of $Z_{j,l}^\pm$ and Lemma \ref{lemA.1}, for $j=1,\ldots,k$, we have
			\be \label{2.9}
			|Z_{j,2}^\pm|=O (\lam ^{-\beta}Z_{x_j^{\pm},\lam }),\quad |Z_{j,l}^\pm|=O(\lam  Z_{x_j^{\pm},\lam }),\quad l=3,\ldots,n,
			\ee
			where $\beta=\frac{\al }{n-2s}$. Combining estimates \eqref{2.9} and Lemma \ref{lemA.2}, we have
			\begin{align*}
				&\int\frac{1}{|z-y|^{n-2s}}\Big|\sum_{j=1}^k(Z_{x_j^+,\lam }^{2_s^*-2}Z_{j,l}^+ +Z_{x_j^-,\lam }^{2_s^*-2}Z_{j,l}^-)\Big| \, \d z\notag\\
				\lesssim&\lam ^{\frac{n+2s}{2}+n_l} \int\frac{1}{|z-y|^{n-2s}}\Big(\sum_{j=1}^k\frac{1}{(1+\lam |z-x_j^+|)^{n+2s}}+\sum_{j=1}^k\frac{1}{(1+\lam |z-x_j^-|)^{n+2s}}\Big)\, \d z\ \notag\\
				\lesssim&  \lam ^{\frac{n-2s}{2}+n_l}\Big(\sum_{j=1}^k\frac{1}{(1+\lam |y-x_j^+|)^{\frac{n-2s}{2}+\tau}}+\sum_{j=1}^k\frac{1}{(1+\lam |y-x_j^-|)^{\frac{n-2s}{2}+\tau}}\Big),
			\end{align*}
			where $n_2=-\beta$, $n_l=1$ for $l=3,\ldots,n$.

			Next, we want to estimate $c_l$, $l=2,\ldots,n$. Multiplying equation \eqref{a.2.1} by $Z_{1,t}^+$ $ (t=2,\ldots,n)$ and integrating, we see that $c_l$ satisfies
			\begin{align}\label{2.11}
				&\sum_{l=2}^n c_l\sum_{j=1}^k\int\big(Z_{x_j^+,\lam }^{2_s^*-2}Z_{j,l}^+ +Z_{x_j^-,\lam }^{2_s^*-2}Z_{j,l}^-\big)Z_{1,t}^+\notag\\
				=&\int ((-\Delta)^s\varphi+V(r,y'')\varphi-(2_s^*-1)K(r,y'')Z_{\bar{r},\bar{h},\bar{y}'',\lam }^{2_s^*-2}\varphi-f)Z_{1,t}^+.
			\end{align}
			It is easy to verify that
			\be \label{2.12}
			\sum_{j=1}^k\int(Z_{x_j^+,\lam }^{2_s^*-2}Z_{j,l}^+ +Z_{x_j^-,\lam }^{2_s^*-2}Z_{j,l}^-)Z_{1,t}^+=
			\left\{
			\begin{aligned}&
				(\bar{c}+o(1))\lam ^{2n_t}, && l=t, \\
				&o(1), && l\neq t,
			\end{aligned}
			\right.
			\ee
			for some constant $\bar{c}>0$.
			
			It follows from Lemma \ref{lemA.1} and \eqref{2.9} that
			\begin{align*}
				&|\langle V(r,y'')\varphi,Z_{1,t}^+\rangle| \notag\\
				\lesssim& \lam^{\frac{n-2s}{2}+n_t}\|\varphi\|_*\int\frac{\eta}{(1+\lam|y-x_1^+|)^{n-2s}}  \times\Big(\sum_{j=1}^{k}\frac{\lam^{\frac{n-2s}{2}}}{(1+\lam|y-x_j^+|)^{\frac{n-2s}{2}+\tau}}+\sum_{j=1}^{k}\frac{\lam^{\frac{n-2s}{2}}}{(1+\lam|y-x_j^-|)^{\frac{n-2s}{2}+\tau}}\Big)\notag\\
				\lesssim&\|\varphi\|_* \lam^{n-2s+n_t} \Big(\int\frac{\eta}{(1+\lam|y-x_1^+|)^{\frac{3}{2}(n-2s)+\tau}}\notag\\&+\sum_{j=2}^{k}\frac{1}{(\lam|x_1^+-x_j^+|)^{\tau}}\int\Big(\frac{\eta}{(1+\lam|y-x_1^+|)^{\frac{3}{2}(n-2s)}}+\frac{\eta}{(1+\lam|y-x_j^+|)^{\frac{3}{2}(n-2s)}}\Big)\notag\\
				&+\sum_{j=1}^{k}\frac{1}{(\lam|x_1^+-x_j^-|)^{\tau}}\int\Big(\frac{\eta}{(1+\lam|y-x_1^+|)^{\frac{3}{2}(n-2s)}}+\frac{\eta}{(1+\lam|y-x_j^-|)^{\frac{3}{2}(n-2s)}}\Big)\Big)\notag\\
				\lesssim &\|\varphi\|_* \int\frac{\eta\lam^{n-2s+n_t+\tau}}{(1+\lam|y-x_1^+|)^{\frac{3}{2}(n-2s)}}\notag\\
				\lesssim&\frac{\lam^{n_t}\|\varphi\|_*}{\lam^{s+\var}},
			\end{align*}
			where $\frac{3(n-2s)}{2}+s-\tau-\var>n$, we use for $|(r,y'')-(r_0,y''_0)|<2\sigma$,
			$$\frac{1}{\lam}\lesssim\frac{1}{1+\lam|y-x^+_1|},$$
			and for $\gamma=\tau=\frac{n-4s-\al}{2(n-2s-\al)}<1$ in Lemma \ref{lemA.4}, $\bar{h}$ satisfying \eqref{1.5},
			\begin{align}\label{a2.13}
				&\sum_{j=2}^{k}\frac{1}{(\lam|x_1^+-x_j^+|)^{\tau}}
				\lesssim\frac{k}{\lam^{(1-\frac{\al}{n-2s})\tau}}\lesssim \lam^{\tau},
			\end{align}
			and
			\begin{align}\label{a2.13.1}
				&\sum_{j=2}^{k}\frac{1}{(\lam|x_1^+-x_j^-|)^{\tau}}
				\lesssim\frac{k}{\lam^{(1-\frac{\al}{n-2s})\tau}}+\frac{1}{\lam^\tau}\lesssim \lam^{\tau}.
			\end{align}
			Similarly, we obtain
			\begin{align*}
				&\Big|\int fZ_{1,t}^+\Big|\notag\\
				\lesssim& \|f\|_{**}\int\frac{\eta\lam ^{\frac{n-2s}{2}+n_t}}{(1+\lam |z-x_1^+|)^{n-2s}}\Big(\sum_{j=1}^{k}\frac{\lam ^{\frac{n+2s}{2}}}{(1+\lam |z-x_j^+|)^{\frac{n+2s}{2}+\tau}}+\sum_{j=1}^{k}\frac{\lam ^{\frac{n+2s}{2}}}{(1+\lam |z-x_j^-|)^{\frac{n+2s}{2}+\tau}}\Big)\, \d z\notag\\
				\lesssim& \lam ^{n_t}\|f\|_{**}\Big(1+\sum_{j=2}^k\frac{1}{(\lam |x_1^+-x_j^+|)^{\sigma}} + \sum_{j=1}^k\frac{1}{(\lam |x_1^+-x_j^-|)^{\sigma}}\Big)\notag\\
				\lesssim&  \lam ^{n_t}\|f\|_{**},
			\end{align*}
			where we used \eqref{A.1} and \eqref{A.2} with $\sigma\in [\frac{n-4s-\alpha}{n-2s-\alpha}, 1)$ such that\begin{align}\label{a2.14}
				&\sum_{j=2}^{k}\frac{1}{(\lam |x_1^+-x_j^+|)^{\sigma}}
				\lesssim\frac{k}{\lam ^{(1-\frac{\al }{n-2s})\sigma}}\leq C,
			\end{align}
			and
			\begin{align}\label{a2.14.1}
				&\sum_{j=1}^{k}\frac{1}{(\lam |x_1^+-x_j^-|)^{\sigma}}
				\lesssim\frac{k}{\lam ^{(1-\frac{\al }{n-2s})\sigma}}+\frac{1}{\lam ^\sigma}\leq C.
			\end{align}
			
			On the other hand, direct calculation gives
			\be \label{2.13}
			\int \Big((-\Delta)^s\varphi-(2_s^*-1)K(r,y'')Z_{\bar{r},\bar{h},\bar{y}'',\lam }^{2_s^*-2}\varphi\Big)Z_{1,t}^+ =O\big(\frac{\lam ^{n_t}\|\varphi\|_*}{\lam ^{s+\var }}\big).
			\ee
			Hence,
			$$
			\text{ RHS of \eqref{2.11} } = O\Big(\lam ^{n_t}\Big(\Big(\frac{\|\varphi\|_*}{\lam ^{s+\var }}\Big)+\|f\|_{**}\Big)\Big),
			$$
			which, together with \eqref{2.11} and \eqref{2.12}, yields that
			\begin{equation*}
				c_l=\frac{1}{\lam ^{n_l}}\big(o(\|\varphi\|_*)+O(\|f\|_{**})\big).
			\end{equation*}
			So,
			\begin{equation*}
				\|\varphi\|_* \lesssim \Big(o(1)+\|f\|_{**}+\frac{\sum_{j=1}^k\big(\frac{1}{(1+\lam |y-x_j^+|)^{\frac{n-2s}{2}+\tau+\theta}}+\frac{1}{(1+\lam |y-x_j^-|)^{\frac{n-2s}{2}+\tau+\theta}}\big)}{\sum_{j=1}^k\big(\frac{1}{(1+\lam |y-x_j^+|)^{\frac{n-2s}{2}+\tau}}+\frac{1}{(1+\lam |y-x_j^-|)^{\frac{n-2s}{2}+\tau}}\big)}\Big),
			\end{equation*}
			which, together with $\|\varphi\|_*=1$, yields that there is $R>0$ such that
			\be \label{2.18}
			\|\lam ^{-\frac{n-2s}{2}}\varphi\|_{L^\infty(B_{R/\lam }(x_j^*))}\geq a>0,
			\ee
			for some $j$ with $x_j^*=x_j^+$ or $x_j^-$. But $\tilde{\varphi}(y)=\lam ^{-\frac{n-2s}{2}}\varphi(\lam ^{-1}y+x_j^*)$ converges uniformly in any compact set to a solution $u$ of
			\be \label{2.19}
			(-\Delta)^s u -(2_s^*-1)U_{0,1}^{2_s^*-2}u=0 \quad \text{ in }\,\Rn,
			\ee
			and $u$ is perpendicular to the kernel of \eqref{2.19}, according to the definition of $\mathbb{H}$. As a consequence, $u=0$, which is a contradiction to \eqref{2.18}.
		\end{proof}
		
		From Lemma \ref{lem2.1}, using the same argument as in the proof of Proposition 4.1 in \cite{del2003two}, we can prove the following result.
		
		\begin{lem}\label{lem2.2}
			There exist $k_0>0$, independent of $k$, such that for $k\geq k_0$ and all $f \in L^\infty(\Rn)$, problem \eqref{a.2.1} has a unique solution $\varphi=\mathcal{L}_k(f)$. Moreover,
			$$
			\|\mathcal{L}_k(f)\|_* \lesssim \|f\|_{**}, \quad |c_l|\lesssim \frac{1}{\lam ^{n_l}}\|f\|_{**},
			$$
			where $n_2=-\beta$, $n_l=1$ for $l=3,\ldots,n$.
		\end{lem}
		
		Next, we consider the perturbation problem for $\varphi \in \mathbb{H}$:
		\begin{align}
			&(-\Delta)^s(Z_{\bar{r},\bar{h},\bar{y}'',\lam }+\varphi)+V(r,y'')(Z_{\bar{r},\bar{h},\bar{y}'',\lam }+\varphi)\notag\\=& K(r,y'')(Z_{\bar{r},\bar{h},\bar{y}'',\lam }+\varphi)_+^{2_s^*-1}+\sum_{l=2}^n c_l\sum_{j=1}^k\Big(Z_{x_j^+,\lam }^{2_s^*-2}Z_{j,l}^+ +Z_{x_j^-,\lam }^{2_s^*-2}Z_{j,l}^-\Big)\quad \text{ in }\,\Rn.\label{2.21}
		\end{align}
		Rewrite \eqref{2.21} as
		\begin{align}
			&(-\Delta)^s\varphi+V(r,y'')\varphi-(2_s^*-1)K(r,y'')Z_{\bar{r},\bar{h},\bar{y}'',\lam }^{2_s^*-2}\varphi\notag\\=& N(\varphi)+l_k 
			+\sum_{l=2}^n c_l\sum_{j=1}^k\Big(Z_{x_j^+,\lam }^{2_s^*-2}Z_{j,l}^+ +Z_{x_j^-,\lam }^{2_s^*-2}Z_{j,l}^-\Big)\quad \text{ in }\,\Rn,\label{2.22}
		\end{align}
		where
		$$
		N(\varphi)=K(r,y'')\Big((Z_{\bar{r},\bar{h},\bar{y}'',\lam }+\varphi)_+^{2_s^*-1}-Z_{\bar{r},\bar{h},\bar{y}'',\lam }^{2_s^*-1}-(2_s^*-1)Z_{\bar{r},\bar{h},\bar{y}'',\lam }^{2_s^*-2}\varphi\Big),
		$$
		and
		\begin{align*}
			l_k=&\Big(K(r,y'')Z_{\bar{r},\bar{h},\bar{y}'',\lam }^{2_s^*-1}-\sum_{j=1}^{k} \big(\eta U_{x_j^+,\lam }^{2_s^*-1} +\eta U_{x_j^-,\lam }^{2_s^*-1}\big)\Big)-V(r,y'')Z_{\bar{r},\bar{h},\bar{y}'',\lam }\notag\\
			&-\sum_{j=1}^k c(n, s) \lim _{\var \rightarrow 0^{+}} \int_{\Rn\backslash B_\var(y)}\Big(\frac{(\eta(x)-\eta(y)) U_{x_j^{+}, \lam}(y)}{|x-y|^{n+2 s}}+\frac{(\eta(x)-\eta(y)) U_{x_j^{-}, \lam}(y)}{|x-y|^{n+2 s}}\Big)\,\d y\notag \\
			= & J_0-J_1-J_2.
		\end{align*}
		
		In order to apply the contraction mapping theorem to prove that equation \eqref{2.22} is uniquely solvable if $\|\varphi\|_*$ is small enough, we need to estimate $N(\varphi)$ and $l_k$, respectively.
		
		\begin{lem}\label{lem2.3}
			If $n> 4s+1$ and $\|\varphi\|_*\leq1$, then
			\begin{equation*}
				\|N(\varphi)\|_{**}\lesssim\lam^{\frac{4s\tau}{n-2s}}\|\varphi\|_{*}^{\min(2_s^*-1,2)}.
			\end{equation*}
		\end{lem}
		
		\begin{proof}
			Note that $K\left(r, y''\right)$ is bounded, it is easy to check that
			$$
			|N(\varphi)| \lesssim \left\{\begin{aligned}&|\varphi|^{2_s^*-1}, && \text { if } 2_s^* \leq 3 \\ &|Z_{\bar{r}, \bar{h}, \bar{y}'', \lam}|^{2_s^*-3} \varphi^2+|\varphi|^{2_s^*-1}, && \text { if } 2_s^*>3.\end{aligned}\right.
			$$
			If $2_s^*\leq3$, we have
			\begin{equation*}
				|N(\varphi)|\lesssim |\varphi|^{2_s^*-1},
			\end{equation*}
			which, together with H\"older inequality, yields that
			\begin{align}\label{b2.23}
				|N(\varphi)|\lesssim&\|\varphi\|_{*}^{2_s^*-1} \lam ^{\frac{n+2s}{2}}\Big(\sum_{j=1}^{k}\frac{1}{(1+\lam |y-x_j^+|)^{\frac{n-2s}{2}+\tau}}+\sum_{j=1}^{k}\frac{1}{(1+\lam |y-x_j^-|)^{\frac{n-2s}{2}+\tau}}\Big)^{2_s^*-1}\notag\\
				\lesssim& \|\varphi\|_{*}^{2_s^*-1} \lam ^{\frac{n+2s}{2}}\Big(\sum_{j=1}^{k}\frac{1}{(1+\lam |y-x_j^+|)^{\frac{n+2s}{2}+\tau}}+\sum_{j=1}^{k}\frac{1}{(1+\lam |y-x_j^-|)^{\frac{n+2s}{2}+\tau}}\Big)\notag\\
				& \times\Big(\sum_{j=1}^{k}\frac{1}{(1+\lam |y-x_j^+|)^{\tau}}+\sum_{j=1}^{k}\frac{1}{(1+\lam |y-x_j^-|)^{\tau}}\Big)^\frac{4s}{n-2s}\notag\\
				\lesssim&\|\varphi\|_{*}^{2_s^*-1} \lam ^{\frac{n+2s}{2}+\frac{4s\tau}{n-2s}}\Big(\sum_{j=1}^{k}\frac{1}{(1+\lam |y-x_j^+|)^{\frac{n+2s}{2}+\tau}}+\sum_{j=1}^{k}\frac{1}{(1+\lam |y-x_j^-|)^{\frac{n+2s}{2}+\tau}}\Big),  \end{align}
			where we used \eqref{a2.13} and \eqref{a2.13.1}.
			
			Therefore,
			\begin{equation*}
				\|N(\varphi)\|_{**}\lesssim\lam^{\frac{4s\tau}{n-2s}} \|\varphi\|_{*}^{2_s^*-1}.
			\end{equation*}
			
			The result for $2_s^*>3$ can be proved in a similar argument. 
		\end{proof}
		
		Next, we estimate $l_k$.
		\begin{lem}\label{lem2.4}
			If $n\geq4$ and $s$ satisfies \eqref{0}, then there is a small constant $\var >0$, such that
			\begin{equation*}
				\|l_k\|_{**}\lesssim\frac{1}{\lam ^{\frac{2s+1-\beta}{2}+\var }}.
			\end{equation*}
		\end{lem}
		
		\begin{proof}
			By symmetry, we can assume that $y \in \Omega_1^+$. Then it follows from \eqref{2.1} and \eqref{2.2} that
			\be \label{b2.24}
			|y-x_1^+|\leq |y-x_1^-|, \quad |y-x_1^+|\leq |y-x_j^+|\leq |y-x_j^-|,\quad j=2,\ldots,k.
			\ee
			
			Firstly, we estimate $J_0$.
			\begin{align*}
				J_0=&K(r,y'')Z_{\bar{r},\bar{h},\bar{y}'',\lam }^{2_s^*-1}-\sum_{j=1}^{k} (\eta U_{x_j^+,\lam }^{2_s^*-1}+\eta U_{x_j^-,\lam }^{2_s^*-1})\\
				=&K(r,y'')\Big(Z_{\bar{r},\bar{h},\bar{y}'',\lam }^{2_s^*-1}-\sum_{j=1}^{k} (\eta U_{x_j^+,\lam }^{2_s^*-1}+\eta U_{x_j^-,\lam }^{2_s^*-1})\Big)+(K(r,y'')-1)\sum_{j=1}^{k} (\eta U_{x_j^+,\lam }^{2_s^*-1}+\eta U_{x_j^-,\lam }^{2_s^*-1})\\
				=&J_{0,1}+J_{0,2}.\end{align*}
			Since $K$ is bounded, we have
			\begin{align*}
				J_{0,1}\lesssim& \Big|\Big(\sum_{j=1}^k (\eta U_{x_j^+,\lam }+\eta U_{x_j^-,\lam })\Big)^{2_s^*-1}-\sum_{j=1}^k (\eta U_{x_j^+,\lam }^{2_s^*-1}+\eta U_{x_j^-,\lam }^{2_s^*-1})\Big|\\
				\lesssim&U_{x_1^+,\lam }^{2_s^*-2}(\sum_{j=2}^k U_{x_j^+,\lam }+\sum_{j=1}^k U_{x_j^-,\lam })+(\sum_{j=2}^kU_{x_j^+,\lam }+\sum_{j=1}^k U_{x_j^-,\lam })^{2_s^*-1}\\
				\lesssim&\frac{\lam ^{\frac{n+2s}{2}}}{(1+\lam |y-x_1^+|)^{4s}} \Big(\sum_{j=2}^{k}\frac{1}{(1+\lam |y-x_j^+|)^{n-2s}}+\sum_{j=1}^{k}\frac{1}{(1+\lam |y-x_j^-|)^{n-2s}}\Big)\\
				& + \lam ^{\frac{n+2s}{2}}\Big(\sum_{j=2}^{k}\frac{1}{(1+\lam |y-x_j^+|)^{n-2s}}+\sum_{j=1}^{k}\frac{1}{(1+\lam |y-x_j^-|)^{n-2s}}\Big)^{2_s^*-1}\\
				=&J_{0,1,1}+J_{0,1,2}.
			\end{align*}
			For the term $J_{0,1,1}$, if $n-2s\geq\frac{n+2s}{2}-\tau$, then it follows from \eqref{b2.24} that
			\begin{align}\notag
				J_{0,1,1}\lesssim&\frac{\lam ^{\frac{n+2s}{2}}}{(1+\lam |y-x_1^+|)^{\frac{n+2s}{2}+\tau}} \Big(\sum_{j=2}^{k}\frac{1}{(1+\lam |y-x_j^+|)^{\frac{n+2s}{2}-\tau}}+\sum_{j=1}^{k}\frac{1}{(1+\lam |y-x_j^-|)^{\frac{n+2s}{2}-\tau}}\Big)\\\notag
				\lesssim& \frac{\lam ^{\frac{n+2s}{2}}}{(1+\lam |y-x_1^+|)^{\frac{n+2s}{2}+\tau}} \Big(\sum_{j=2}^{k}\frac{1}{(\lam |x_1^+-x_j^+|)^{\frac{n+2s}{2}-\tau}}+\sum_{j=1}^{k}\frac{1}{(\lam |x_1^+-x_j^-|)^{\frac{n+2s}{2}-\tau}}\Big)\\\label{b.2.21}
				\lesssim& \frac{1}{\lam ^{\frac{2s+1-\beta}{2}+\var }}\frac{\lam ^{\frac{n+2s}{2}}}{(1+\lam |y-x_1^+|)^{\frac{n+2s}{2}+\tau}},
			\end{align}
			where we  used for $\gamma>1$ in Lemma \ref{lemA.4}, $\bar{h}$ satisfying \eqref{1.5},
			\be \label{b2.14}
			\sum_{j=2}^{k}\frac{1}{(\lam |x_1^+-x_j^+|)^{\gamma}}\lesssim\frac{k^\gamma}{\lam ^{(1-\frac{\al }{n-2s})\gamma}}=O\Big(\frac{1}{\lam ^{\frac{2s\gamma}{n-2s}}}\Big),
			\ee
			\be \label{b2.14.1}
			\sum_{j=1}^{k}\frac{1}{(\lam |x_1^+-x_j^-|)^{\gamma}}\lesssim\frac{k^\gamma}{\lam ^{(1-\frac{\al }{n-2s})\gamma}}+\frac{1}{(\lam \bar{h})^\gamma} =O\Big(\frac{1}{\lam ^{\frac{2s\gamma}{n-2s}}}\Big),
			\ee
			and the fact $$\frac{2s}{n-2s}(\frac{n+2 s}{2}-\tau)>\frac{2 s+1-\beta}{2} .$$
			If $n-2s<\frac{n+2s}{2}-\tau$, then $4s>\frac{n+2s}{2}+\tau$. Similar to \eqref{b.2.21}, we have	
			\begin{align*}
				J_{0,1,1}\lesssim&\frac{\lam ^{\frac{n+2s}{2}}}{(1+\lam |y-x_1^+|)^{\frac{n+2s}{2}+\tau}} \Big(\sum_{j=2}^{k}\frac{1}{(1+\lam |y-x_j^+|)^{n-2s}}+\sum_{j=1}^{k}\frac{1}{(1+\lam |y-x_j^-|)^{n-2s}}\Big)\\
				\lesssim& \frac{\lam ^{\frac{n+2s}{2}}}{(1+\lam |y-x_1^+|)^{\frac{n+2s}{2}+\tau}} \Big(\sum_{j=2}^{k}\frac{1}{(\lam |x_1^+-x_j^+|)^{n-2s}}+\sum_{j=1}^{k}\frac{1}{(\lam |x_1^+-x_j^-|)^{n-2s}}\Big)\\
				\lesssim& \frac{1}{\lam ^{\frac{2s+1-\beta}{2}+\var }}\frac{\lam ^{\frac{n+2s}{2}}}{(1+\lam |y-x_1^+|)^{\frac{n+2s}{2}+\tau}},
			\end{align*}
			where we used the fact $2s>\frac{2s+1-\beta}{2}$.
			
			As for $J_{0,1,2}$, using the H\"older inequality, we have
			\begin{align*}
				&\Big(\sum_{j=2}^{k}\frac{1}{(1+\lam |y-x_j^+|)^{n-2s}}+\sum_{j=1}^{k}\frac{1}{(1+\lam |y-x_j^-|)^{n-2s}}\Big)^{2_s^*-1}\\
				\lesssim&
				\sum_{j=2}^{k}\frac{1}{(1+\lam |y-x_j^+|)^{\frac{n+2s}{2}+\tau}}\Big(\sum_{j=2}^{k}\frac{1}{(\lam |x_1^+-x_j^+|)^{\frac{n+2s}{4s}(\frac{n-2s}{2}-\frac{n-2s}{n+2s}\tau)}}\Big)^{\frac{4s}{n-2s}}\\
				&+\sum_{j=1}^{k}\frac{1}{(1+\lam |y-x_j^-|)^{\frac{n+2s}{2}+\tau}}\Big(\sum_{j=1}^{k}\frac{1}{(\lam |x_1^+-x_j^-|)^{\frac{n+2s}{4s}(\frac{n-2s}{2}-\frac{n-2s}{n+2s}\tau)}}\Big)^{\frac{4s}{n-2s}}\\
				\lesssim& \Big(\sum_{j=2}^{k}\frac{1}{(1+\lam |y-x_j^+|)^{\frac{n+2s}{2}+\tau}}+\sum_{j=1}^{k}\frac{1}{(1+\lam |y-x_j^-|)^{\frac{n+2s}{2}+\tau}}\Big)\Big(\frac{1}{\lam }\Big)^{\frac{2s+1-\beta}{2}+\var },
			\end{align*}
			since we can check that $$\frac{n+2s}{4s}\Big(\frac{n-2s}{2}-\frac{n-2s}{n+2s}\tau\Big)>1$$ and \be\label{1}\frac{2s}{n-2s}\frac{n+2s}{n-2s}\Big(\frac{n-2s}{2}-\frac{n-2s}{n+2s}\tau\Big)>\frac{2s+1-\beta}{2}.\ee
			Thus,
			\begin{equation}\label{J01}
				\|J_{0,1}\|_{**}=O\Big(\frac{1}{\lam ^{\frac{2s+1-\beta}{2}+\var }}\Big).
			\end{equation}
			
			Secondly, we will estimate $J_{0,2}$. 
			
			In the region $|(r,y'')-(r_0,y_0'')|\leq\frac{1}{\lam^{\frac{2s+1-\beta}{4}+\var}}$, we have
			
			\begin{align*}
				|J_{0,2}|  =&\Big|\Big(\sum_{|\alpha|=2} \frac{1}{2} D^\alpha K(r_0, y_0'')(y-y_0)^\alpha+o(|y-y_0|^2)\Big) \sum_{j=1}^k (\eta U_{x_j^+,\lam }^{2_s^*-1}+\eta U_{x_j^-,\lam }^{2_s^*-1})\Big| \\
				\lesssim&\frac{1}{\lam^{\frac{2s+1-\beta}{2}+\var}} \Big(\sum_{j=1}^{k}\frac{|\eta|\lam^{\frac{n+2s}{2}}}{(1+\lam |y-x_j^+|)^{n+2s}}+\sum_{j=1}^{k}\frac{|\eta|\lam^{\frac{n+2s}{2}}}{(1+\lam |y-x_j^-|)^{n+2s}}\Big) \\
				\lesssim&\frac{1}{\lam^{\frac{2s+1-\beta}{2}+\var}} \Big(\sum_{j=1}^{k}\frac{\lam^{\frac{n+2s}{2}}}{(1+\lam |y-x_j^+|)^{\frac{n+2s}{2}+\tau}}+\sum_{j=1}^{k}\frac{\lam^{\frac{n+2s}{2}}}{(1+\lam |y-x_j^-|)^{\frac{n+2s}{2}+\tau}}\Big).
			\end{align*}

			On the other hand, in the region $\frac{1}{\lam^{\frac{2s+1-\beta}{4}+\var}} \leq|(r, y'')-(r_0, y_0'')| \leq 2 \delta$, we conclude from \eqref{1.5} that
			$$
			|y-x_j^\pm| \geq|(r, y'')-(r_0, y_0'')|-|(r_0, y_0'')-(\bar{r}, \bar{y}'')| \geq \frac{1}{2 \lam^{\frac{2s+1-\beta}{4}+\var}},
			$$
			which leads to $\frac{1}{1+\lam|y-x_j^\pm|} \lesssim \frac{1}{\lam^{1+\frac{2s+1-\beta}{4}+\var}}$, then
			\begin{align*}
				|J_{0,2}|\lesssim&\sum_{j=1}^{k}\frac{|\eta|\lam^{\frac{n+2s}{2}}}{(1+\lam |y-x_j^+|)^{n+2s}}+\sum_{j=1}^{k}\frac{|\eta|\lam^{\frac{n+2s}{2}}}{(1+\lam |y-x_j^-|)^{n+2s}}\\
				\lesssim&\frac{1}{\lam^{\frac{2s+1-\beta}{2}+\var}}\sum_{j=1}^{k}\frac{\lam^{\frac{2s+1-\beta}{2}+\var}\lam^{\frac{n+2s}{2}}}{(1+\lam |y-x_j^+|)^{n+2s}}+\sum_{j=1}^{k}\frac{\lam^{\frac{2s+1-\beta}{2}+\var}\lam^{\frac{n+2s}{2}}}{(1+\lam |y-x_j^-|)^{n+2s}}\\
				\lesssim&\frac{1}{\lam^{\frac{2s+1-\beta}{2}+\var}}\sum_{j=1}^{k}\frac{\lam^{\frac{n+2s}{2}}}{(1+\lam |y-x_j^+|)^{\frac{n+2s}{2}+\tau}}+\sum_{j=1}^{k}\frac{\lam^{\frac{n+2s}{2}}}{(1+\lam |y-x_j^-|)^{\frac{n+2s}{2}+\tau}},
			\end{align*}
			where we used the fact $$\Big(\frac{n+2s}{2}-\tau\Big)\Big(1+\frac{2s+1-\beta}{4}+\var\Big)>\frac{2s+1-\beta}{2}+\var.$$
			Thus,\begin{equation}\label{J02}
				\|J_{0,2}\|_{**}=O\Big( \frac{1}{\lam ^{\frac{2s+1-\beta}{2}+\var }}\Big).
			\end{equation}
			Combining \eqref{J01} and \eqref{J02}, we have $$\|J_0\|_{* *}=O\Big(\frac{1}{\lam^{\frac{2s+1-\beta}{2}+\var }}\Big).$$
			
			For $J_1$, noting that $\frac{1}{\lam}\leq \frac{C}{1+\lam|y-x_j^\pm|}$ when $|(r,y'')-(r_0,y_0'')| < 2\sigma$, we have
			\begin{align*}
				|J_1|  \lesssim& \sum_{j=1}^k \frac{\eta\lam^{\frac{n-2s}{2}}}{(1+\lam|y-x_j^+|)^{n-2s}} +\sum_{j=1}^k \frac{\eta\lam^{\frac{n-2s}{2}}}{(1+\lam|y-x_j^-|)^{n-2s}}\\
				\lesssim& \frac{1}{\lam^{\frac{2s+1-\beta}{2}+\var}} \Big(\sum_{j=1}^{k}\frac{\lam^{\frac{n+2s}{2}}}{\lam^{\frac{2s-1+\beta}{2}-\var}(1+\lam|y-x_j^+|)^{n-2s}}+\sum_{j=1}^{k}\frac{\lam^{\frac{n+2s}{2}}}{\lam^{\frac{2s-1+\beta}{2}-\var}(1+\lam|y-x_j^-|)^{n-2s}}\Big)\\
				\lesssim& \frac{1}{\lam^{\frac{2s+1-\beta}{2}+\var}} \Big(\sum_{j=1}^{k}\frac{\lam^{\frac{n+2s}{2}}}{(1+\lam|y-x_j^+|)^{\frac{n+2s}{2}+\tau}}+\sum_{j=1}^{k}\frac{\lam^{\frac{n+2s}{2}}}{(1+\lam|y-x_j^-|)^{\frac{n+2s}{2}+\tau}}\Big) .
			\end{align*}
			Since we can check that \be\label{2}n-2s+\frac{2s-1+\beta}{2}-\var>\frac{n+2s}{2}+\tau.\ee Thus,
			$$\|J_1\|_{* *}= O\Big(\frac{1}{\lam^{\frac{2s+1-\beta}{2}+\var }}\Big).$$
			For $J_2$, we have
			\begin{align*}
				J_2 =&\sum_{j=1}^{k}c(n,s)\Big(\lim_{\var \to 0^+} \int_{B_{\frac{\sigma}{4}}(x) \backslash B_\var(x)}\frac{\big(\eta(x)-\eta(y)\big)\big(U_{x_j^+,\lam}(y)+U_{x_j^-,\lam}(y)\big)}{|x-y|^{n+2s}} dy \\
				& + \int_{\Rn\backslash B_{\frac{\sigma}{4}}(x)} \frac{\big(\eta(x)-\eta(y)\big)\big(U_{x_j^+,\lam}(y)+U_{x_j^-,\lam}(y)\big)}{|x-y|^{n+2s}} \,\d y\Big)\\
				=&\sum_{j=1}^{k}c(n,s)(J_{21}+J_{22}).
			\end{align*}
			Now we estimate the term $J_{21}$. From the definition of the function $\eta$, we have $\eta(x)-\eta(y)=0$ when $x,y\in B_\sigma((r_0,y''_0))$ or $x,y\in\Rn\backslash B_{2\sigma}((r_0,y''_0))$. So, $J_{21}\neq 0$ only if $B_{\frac{\sigma}{4}}(x)\subset B_{\frac{5}{2}\sigma}((r_0,y''_0))\backslash B_{\frac{1}{2}\sigma}((r_0,y''_0))$. Since $x_j^{\pm}\to(r_0,y''_0)$, we only consider the case
			\begin{equation*}
				\frac{\sigma}{2} \leq|x-x_j^{\pm}|\leq \frac{5\sigma}{2} \quad \text{ and } \quad \frac{\sigma}{4} \leq|y-x_j^{\pm}|\leq \frac{11\sigma}{4}.
			\end{equation*}
			Then we have $\frac{1}{10}|x-x_j^{\pm}|\leq|y-x_j^{\pm}|\leq\frac{11}{2}|x-x_j^{\pm}|$.
			
			Furthermore,
			\begin{equation*}
				\begin{aligned}
					&\lim_{\var \to 0^+} \int_{B_{\frac{\sigma}{4}}(x) \backslash B_\var(x)}\frac{(\eta(x)-\eta(y)\big)U_{x_j^{\pm},\lam}(y)}{|x-y|^{n+2s}}\, \d y\\
					=&\lim_{\var \to 0^+} \int_{B_{\frac{\sigma}{4}}(x) \backslash B_\var(x)}\frac{\langle\nabla \eta(x),x-y\rangle U_{x_j^{\pm},\lam}(y)}{|x-y|^{n+2s}} \,\d y
					+O\Big(\lim_{\var \to 0^+} \int_{B_{\frac{\sigma}{4}}(x) \backslash B_\var(x)}\frac{U_{x_j^{\pm},\lam}(y)}{|x-y|^{n+2s-2}} \,\d y\Big).
				\end{aligned}
			\end{equation*}
			By the mean value theorem and symmetry, there exists a constant $0<\theta<1$ such that
			\begin{align}
				&\Big|\lim_{\var \to 0^+} \int_{B_{\frac{\sigma}{4}}(x) \backslash B_\var(x)}\frac{\langle\nabla \eta(x),x-y\rangle U_{x_j^{\pm},\lam}(y)}{|x-y|^{n+2s}}\,\d y\Big|\notag\\
				=&\Big|\lim_{\var \to 0^+} \int_{B_{\frac{\sigma}{4}}(0) \backslash B_\var(0)}\frac{\langle\nabla \eta(x),z\rangle}{|z|^{n+2s}} \frac{\lam^{\frac{n-2s}{2}}}{(1+\lam^{2}|z+x-x^{\pm}_j|^{2})^{\frac{n-2s}{2}}}\,\d z\Big|\notag\\
				\lesssim& \lam^{\frac{n-2s}{2}+1}\int_{B_{\frac{\sigma}{4}}(0)}\frac{|\nabla \eta(x)|}{|z|^{n+2s-2}} \frac{1}{(1+\lam|(2\theta-1)z+x-x^{\pm}_j|)^{n-2s+1}}\,\d z\notag\\
				\lesssim& \frac{1}{\lam^{\frac{2s+1-\beta}{2}+\var}}\frac{\lam^{\frac{n+2s}{2}}}{\lam^{2s-1+\frac{-2s-1+\beta}{2}-\var}(1+\lam|x-x^{\pm}_j|)^{n-2s+1}}\notag\\
				\lesssim&\label{2.5} \frac{1}{\lam^{\frac{2s+1-\beta}{2}+\var}}\frac{\lam^{\frac{n+2s}{2}}}{(1+\lam|x-x^{\pm}_j|)^{\frac{n+2s}{2}+\tau}},
			\end{align}
			where we used the fact that $$|(2\theta-1)z+x-x^+_j|\geq|x-x^{\pm}_j|-|(2\theta-1)z|\geq\frac{1}{10}|x-x^{\pm}_j|\quad\text{ for }z\in B_{\frac{\sigma}{4}(0)},$$ $$\frac{1}{\lam }\leq \frac{C}{1+\lam |y-x_j^\pm|}\quad\text{ when }|(r,y'')-(r_0,y_0'')| < 2\sigma,$$ and $$n+\frac{-2s-1+\beta}{2}-\var>\frac{n+2s}{2}+\tau.$$
			
			Similar to \eqref{2.5}, we can obtain
			\begin{align*}
				&\Big|\lim_{\var \to 0^+} \int_{B_{\frac{\sigma}{4}}(x) \backslash B_\var(x)}\frac{U_{x_j^{\pm},\lam}(y)}{|x-y|^{n+2s-2}}\,\d y\Big|\\
				\lesssim&\int_{B_{\frac{\sigma}{4}}(x)}\frac{1}{|x-y|^{n+2s-2}}\frac{\lam^{\frac{n-2s}{2}}}{(1+\lam|y-x^{\pm}_j|)^{n-2s}}\,\d y\\
				\lesssim& \frac{1}{\lam^{\frac{2s+1-\beta}{2}+\var}}\int_{B_{\frac{\sigma}{4}}(x)}
				\frac{1}{|x-y|^{n+2s-2}}\frac{\lam^{\frac{n+2s}{2}}}{(1+\lam|y-x^{\pm}_j|)^{n-\frac{2s+1-\beta}{2}-\var}}\,\d y\\
				\lesssim& \frac{1}{\lam^{\frac{2s+1-\beta}{2}+\var}}\frac{\lam^{\frac{n+2s}{2}}}{(1+\lam|x-x^{\pm}_j|)^{\frac{n+2s}{2}+\tau}}.
			\end{align*}
			Hence,
			\begin{align*}
				|J_{21}|=&\Big|\lim_{\var \to 0^+} \int_{B_{\frac{\sigma}{4}}(x) \backslash B_\var(x)}\frac{\big(\eta(x)-\eta(y)\big)\big(U_{x_j^+,\lam}(y)+U_{x_j^-,\lam}(y)\big)}{|x-y|^{n+2s}}\,\d y\Big|\\
				\lesssim& \frac{1}{\lam^{\frac{2s+1-\beta}{2}+\var}}\Big(\frac{\lam^{\frac{n+2s}{2}}}{(1+\lam|x-x^+_j|)^{\frac{n+2s}{2}+\tau}}+\frac{\lam^{\frac{n+2s}{2}}}{(1+\lam|x-x^-_j|)^{\frac{n+2s}{2}+\tau}}\Big).
			\end{align*}
			As for the term $J_{22}$, we divide it into the following three cases:
			
			In the case of $x\in B_\sigma(x^+_j)$, we have
			\begin{align*}
				&\Big|\int_{\Rn\backslash B_{\frac{\sigma}{4}}(x)}\frac{\big(\eta(x)-\eta(y)\big)U_{x_j^+,\lam}(y)}{|x-y|^{n+2s}}\,\d y\Big|\\
				\lesssim&\int_{\Rn\backslash \big(B_{\frac{\sigma}{4}}(x)\cup B_\sigma((r_0,y''_0))\big)} \frac{1}{|x-y|^{n+2s}}\frac{\lam^{\frac{n-2s}{2}}}{(1+\lam|y-x^+_j|)^{n-2s}}\,\d y\\
				\lesssim& \frac{1}{\lam^{2s}}\frac{\lam^{\frac{n+2s}{2}}}{(1+\lam|x-x^+_j|)^{n-2s}}
				\int_{\Rn\backslash B_{\frac{\sigma}{4}}(x)}\frac{1}{|x-y|^{n+2s}}\,\d y\\
				\lesssim& \frac{1}{\lam^{\frac{2s+1-\beta}{2}+\var}}\frac{\lam^{\frac{n+2s}{2}}}{(1+\lam|x-x^+_j|)^{\frac{n+2s}{2}+\tau}}.
			\end{align*}
			
			In the case of $x\in B_{3\sigma}(x^+_j)\textbackslash B_\sigma(x^+_j)$, using Lemma \ref{lemC.1}, then there holds
			\begin{align*}
				&\Big|\int_{\Rn\backslash B_{\frac{\sigma}{4}}(x)} \frac{(\eta(x)-\eta(y))U_{x_j^+,\lam}(y)}{|x-y|^{n+2s}}\,\d y\Big|\\
				\lesssim&\int_{\Rn\backslash B_{\frac{\sigma}{4}}(x)} \frac{1}{|x-y|^{n+2s}} \frac{\lam^{\frac{n-2s}{2}}}{(1+\lam|y-x^+_j|)^{n-2s}}\,\d y\\
				\lesssim& \lam^{\frac{n+2s}{2}}\int_{\Rn\backslash B_{\frac{\sigma\lam}{4}}(\lam x)} \frac{1}{|z-\lam x|^{n+2s}} \frac{1}{(1+|z-\lam x^+_j|)^{n-2s}}\,\d z\\
				\lesssim&\lam^{\frac{n+2s}{2}}\Big(\frac{1}{(\lam|x-x^+_j|)^n}+\frac{1}{\lam^{2s}}\frac{1}{(\lam|x-x^+_j|)^{n-2s}}\Big)\\
				\lesssim& \frac{1}{\lam^{\frac{2s+1-\beta}{2}+\var}}\frac{\lam^{\frac{n+2s}{2}}}{(1+\lam|x-x^+_j|)^{\frac{n+2s}{2}+\tau}},
			\end{align*}
			where we used $\frac{C_1}{1+\lam|x-x_j^+|}\leq\frac{1}{\lam}\leq\frac{C_2}{1+\lam|x-x_j^+|}$ for $\sigma\leq|x -x^+_j|\leq3\sigma$.
			
			In the case of $x\in \Rn\textbackslash B_{3\sigma}(x^+_j)$. Note that $|x-y|\geq|x-x^+_j|-|y-x^+_j|\geq\frac{1}{3}|x-x^+_j|$ when $|x-x^+_j|>3\sigma$ and $|y-x^+_j|\leq2\sigma$, then we have
			\begin{align*}
				&\Big|\int_{\Rn\backslash B_{\frac{\sigma}{4}}(x)} \frac{(\eta(x)-\eta(y))U_{x_j^+,\lam}(y)}{|x-y|^{n+2s}}\,\d y\Big|\\
				\lesssim&\int_{B_{2\sigma(x_j^+)}} \frac{1}{|x-y|^{n+2s}} \frac{\lam^{\frac{n-2s}{2}}}{(1+\lam|y-x^+_j|)^{n-2s}}\,\d y\\
				\lesssim& \frac{1}{\lam^{\frac{n-2s}{2}}} \int_{B_{2\sigma(x_j^+)}} \frac{1}{|x-y|^{n+2s}} \frac{1}{(|y- x^+_j|)^{n-2s}}\,\d y\\
				\lesssim& \frac{1}{\lam^{\frac{2s+1-\beta}{2}+\var}}\frac{\lam^{\frac{n+2s}{2}}}{(1+\lam|x-x^+_j|)^{\frac{n+2s}{2}+\tau}},
			\end{align*}
			where we used $\frac{1}{|x-y|}\lesssim\frac{\lam}{1+\lam|x-x^+_j|}$.

			Hence,
			\begin{align*}
				\Big|\int_{\Rn\backslash B_{\frac{\sigma}{4}}(x)} \frac{\big(\eta(x)-\eta(y)\big)U_{x_j^+,\lam}(y)}{|x-y|^{n+2s}}\,\d y\Big|
				\lesssim \frac{1}{\lam^{\frac{2s+1-\beta}{2}+\var}}\frac{\lam^{\frac{n+2s}{2}}}{(1+\lam|x-x^+_j|)^{\frac{n+2s}{2}+\tau}}.
			\end{align*}
			
			Similarly, we can deduce that
			\begin{align*}
				\Big|\int_{\Rn\backslash B_{\frac{\sigma}{4}}(x)} \frac{(\eta(x)-\eta(y))U_{x_j^-,\lam}(y)}{|x-y|^{n+2s}} \,\d y\Big|
				\lesssim \frac{1}{\lam^{\frac{2s+1-\beta}{2}+\var}}\frac{\lam^{\frac{n+2s}{2}}}{(1+\lam|x-x^-_j|)^{\frac{n+2s}{2}+\tau}}.
			\end{align*}
			So, we obtain
			\begin{equation*}
				\begin{aligned}
					\|J_{2}\|_{**}\lesssim \frac{1}{\lam^{\frac{2s+1-\beta}{2}+\var}}.
				\end{aligned}
			\end{equation*}
			
			As a result, we have proved that
			\begin{equation*}
				\|l_k\|_{**}\lesssim \frac{1}{\lam^{\frac{2s+1-\beta}{2}+\var}}.
			\end{equation*}
		\end{proof}
		\begin{rem}
			The assumption \eqref{0} guarantees the validity of \eqref{1} and \eqref{2}, which ensures the existence of $\nu$.
		\end{rem}
		The solvability theory for the linearized problem \eqref{2.22} can be provided in the following:
		\begin{prop}\label{prop2.5}
			Suppose that $n\geq 4$ and $s$ satisfies \eqref{0}, there exists a positive large integer $k_0$, such that for all $k\geq k_0$ and $\lam  \in [L_0 k^{\frac{n-2s}{n-4s-\al }},L_1 k^{\frac{n-2s}{n-4s-\al }}]$, $(\bar{r},\bar{h},\bar{y}'')$ satisfies \eqref{1.5}, where $L_0,L_1$ are fixed constants, problem \eqref{2.21} has a unique solution $\varphi=\varphi_{\bar{r},\bar{h},\bar{y}'',\lam }$ satisfying
			\be \label{2.28}
			\|\varphi\|_{*} \lesssim \frac{1}{\lam ^{\frac{2s+1-\beta}{2}+\var }}, \quad |c_l|\lesssim \frac{1}{\lam ^{\frac{2s+1-\beta}{2}+n_l+\var }},
			\ee
			where $\var >0$ is a small constant.
		\end{prop}
		\begin{proof}
			We first denote
			$$
			\mathbb{E}=\Big\{u:u \in C(\Rn)\cap\mathbb{H},~\|u\|_*\leq \frac{1}{\lam ^{\frac{2s+1-\beta}{2}}}\Big\}.
			$$
			Then \eqref{2.22} is equivalent to 
			\be \label{2.30}
			\varphi=\mathcal{A}(\varphi)=\mathcal{L}_k(N(\varphi))+\mathcal{L}_k(l_k),
			\ee
			where $\mathcal{L}_k$ is the linear bounded operator defined in Lemma \ref{lem2.2}. We will prove that $\mathcal{A}$ is a contraction map from $\mathbb{E}$ to $\mathbb{E}$. 
			
			For any $\varphi \in \mathbb{E}$, by Lemma \ref{lem2.2}, Lemma \ref{lem2.3} and Lemma \ref{lem2.4}, we have
			\begin{equation*}
				\begin{aligned}
					\|\mathcal{A}\|_* \lesssim&\|\mathcal{L}_k(N(\varphi))\|_*+C\|\mathcal{L}_k(l_k)\|_*\\
					\lesssim&(\|N(\varphi)\|_{**}+\|l_k\|_{**})\\
					\lesssim&\lam^{\frac{4s\tau}{n-2s}} \|\varphi\|_{*}^{\min(2_s^*-1,2)}+\frac{1}{\lam ^{{\frac{2s+1-\beta}{2}}+\var }}
					\lesssim\frac{1}{\lam ^{{\frac{2s+1-\beta}{2}}}},
				\end{aligned}
			\end{equation*}
			since $\frac{4s}{n-2s}\tau<\min(s,(2_s^*-2)s)$. Hence,  $\mathcal{A}$ maps $\mathbb{E}$ to $\mathbb{E}$.
			
			On the other hand, for all $\varphi_1,\varphi_2 \in \mathbb{E}$, we have
			$$
			\|\mathcal{A}(\varphi_1)-\mathcal{A}(\varphi_2)\|_*=\|\mathcal{L}_k(N(\varphi_1))-\mathcal{L}_k(N(\varphi_2))\|_* \lesssim\|N(\varphi_1)-N(\varphi_2)\|_{**}.
			$$
			If $2_s^*\leq3$, using H\"older inequality like \eqref{b2.23}, then we have
			\begin{align*}
				|N(\varphi_1)-(N(\varphi_2))| \lesssim&(|\varphi_1|^{2_s^*-2}+|\varphi_2|^{2_s^*-2})|\varphi_1-\varphi_2|\\
				\lesssim&(\|\varphi_1\|_*^{2_s^*-2}+\|\varphi_2\|_*^{2_s^*-2})\|\varphi_1-\varphi_2\|_*\\
				& \times\Big(\sum_{j=1}^{k}\frac{\lam ^{\frac{n-2s}{2}}}{(1+\lam |y-x_j^+|)^{\frac{n-2s}{2}+\tau}}+\sum_{j=1}^{k}\frac{\lam ^{\frac{n-2s}{2}}}{(1+\lam |y-x_j^-|)^{\frac{n-2s}{2}+\tau}}\Big)^{2_s^*-1}\\
				\lesssim&(\|\varphi_1\|_*^{2_s^*-2}+\|\varphi_2\|_*^{2_s^*-2})\|\varphi_1-\varphi_2\|_*\\
				& \times\Big(\sum_{j=1}^{k}\frac{\lam ^{\frac{n+2s}{2}}}{(1+\lam |y-x_j^+|)^{\frac{n+2s}{2}+\tau}}+\sum_{j=1}^{k}\frac{\lam ^{\frac{n+2s}{2}}}{(1+\lam |y-x_j^-|)^{\frac{n+2s}{2}+\tau}}\Big).
			\end{align*}
			Thus,
			$$
			\|\mathcal{A}(\varphi_1)-\mathcal{A}(\varphi_2)\|_* \lesssim(\|\varphi_1\|_*^{2_s^*-2}+\|\varphi_2\|_*^{2_s^*-2})\|\varphi_1-\varphi_2\|_*\leq \frac{1}{2}\|\varphi_1-\varphi_2\|_*.
			$$
			Therefore, $\mathcal{A}$ is a contraction map from $\mathbb{E}$ to $\mathbb{E}$. The proof for the case $2_s^*>3$ is similar.
			
			By the contraction mapping theorem, there exists a unique $\varphi=\varphi_{\bar{r},\bar{h},\bar{y}'',\lam } \in \mathbb{E}$ such that \eqref{2.30} holds. Furthermore, according to Lemma \ref{lem2.2}, Lemma \ref{lem2.3} and Lemma \ref{lem2.4}, we deduce
			\begin{equation*}
				\|\varphi\|_* \leq \|\mathcal{L}_k(N(\varphi))\|_*+\|\mathcal{L}_k(l_k)\|_* \lesssim(\|N(\varphi)\|_{**}+\|l_k\|_{**})\lesssim \frac{1}{\lam ^{{\frac{2s+1-\beta}{2}}+\var }}
			\end{equation*}
			and
			\begin{equation*}
				|c_l|\lesssim\frac{1}{\lam ^{n_l}}\|N(\varphi)+l_k\|_{**} \lesssim \frac{1}{\lam ^{{\frac{2s+1-\beta}{2}}+n_l+\var }},
			\end{equation*}
			for $l=2,\ldots,n$.
		\end{proof}
		\section{Expansions for the energy functional}\label{sec:3}
		This section is devoted to computing the partial derivatives of the energy functional  $I(Z_{\bar{r},\bar{h},\bar{y}'',\lam })$ about $\lam$, $\bar{r}$ and $y_j''$ $(j=4,\ldots,n)$, respectively.
		We first recall that
		$$
		I(u)=\frac{1}{2} \int\big(|(-\Delta)^{\frac{s}{2}} u|^2+V(y)u^2\big)\,\d y-\frac{1}{2_s^*} \int K(y)(u)_+^{2_s^*}\,\d y .
		$$
		\begin{lem}\label{lemB.1}
			If $n>4s+1$, then
			\begin{equation*}
				\begin{aligned}
					\frac{\pa  I(Z_{\bar{r},\bar{h},\bar{y}'',\lam })}{\pa  \lam }
					=& 2k\Big(-\frac{sB_1}{\lam ^{2s+1}} + \sum_{j=2}^{k}\frac{B_2}{\lam ^{n-2s+1}|x_j^+-x_1^+|^{n-2s}} 
					+\sum_{j=1}^{k}\frac{B_2}{\lam ^{n-2s+1}|x_j^--x_1^+|^{n-2s}} \\&+ O\Big(\frac{1}{\lam ^{2s+1+\var }}\Big)\Big),
				\end{aligned}
			\end{equation*}
			where $B_1$ and $B_2$ are two positive constants.
		\end{lem}
		
		\begin{proof}
			Note that for any $\varepsilon>0$, $Z_{i,2}^\pm=O(\lam ^{-1} Z_{x_i^\pm,\lam })$ in $B^c_\varepsilon(x_i^\pm)$. A direct computation leads to
			\begin{align*}
				\frac{\pa  I(Z_{\bar{r},\bar{h},\bar{y}'',\lam })}{\pa  \lam }
				=&\frac{\pa  I(Z_{\bar{r},\bar{h},\bar{y}'',\lam }^{*})}{\pa  \lam }+O\Big(\frac{k}{\lam  ^{2 s+1+\varepsilon}}\Big)\\
				=& \int V(y)Z_{\bar{r},\bar{h},\bar{y}'',\lam }^*\frac{\pa  Z_{\bar{r},\bar{h},\bar{y}'',\lam }^{*}}{\pa  \lam }+\int(1-K(r,y''))(Z_{\bar{r},\bar{h},\bar{y}'',\lam }^*)^{2_s^*-1}\frac{\pa  Z_{\bar{r},\bar{h},\bar{y}'',\lam }^{*}}{\pa  \lam }\\&-\int\Big((Z_{\bar{r},\bar{h},\bar{y}'',\lam }^*)^{2_s^*-1}-\sum_{j=1}^{k}U_{x_j^+,\lam }^{2_s^*-1}-\sum_{j=1}^{k}U_{x_j^-,\lam }^{2_s^*-1}\Big)\frac{\pa  Z_{\bar{r},\bar{h},\bar{y}'',\lam }^{*}}{\pa  \lam }+O\Big(\frac{k}{\lam  ^{2 s+1+\varepsilon}}\Big)\\
				=&I_1+I_2-I_3+O\Big(\frac{k}{\lam ^{2 s+1+\var}}\Big).
			\end{align*}
			By Lemma \ref{lemA.1}, we can check that
			\begin{align*}
				{I}_1=&V(r_0,y_0'')\int Z_{\bar{r},\bar{h},\bar{y}'',\lam}^{*}\frac{\pa Z_{\bar{r},\bar{h},\bar{y}'',\lam}^{*}}{\pa \lam}+\int (V(y)-V(r_0,y_0''))Z_{\bar{r},\bar{h},\bar{y}'',\lam}^{*}\frac{\pa Z_{\bar{r},\bar{h},\bar{y}'',\lam}^{*}}{\pa \lam}\\
				=&2k\Big(V(r_0,y_0'')\int U_{x_1^+,\lam }\frac{\pa  U_{x_1^+,\lam }}{\pa  \lam }+V(r_0,y_0'')\int \frac{\pa  U_{x_1^+,\lam }}{\pa  \lam }\Big(\sum_{j=2}^{k}U_{x_j^+,\lam }+\sum_{j=1}^{k}U_{x_j^-,\lam }\Big)
				+O\Big(\frac{1}{\lam ^{2s+1+\var }}\Big)\Big)\\
				=&2k\Big(\frac{V(r_0,y_0'')}{2}\frac{\pa }{\pa  \lam }\int U^2_{x_1^+,\lam } +O\Big(\frac{1}{\lam ^{2s+\beta}}\Big(\sum_{j=2}^{k} \frac{1}{(\lam |x_j^+-x_1^+|)^{n-4s-\theta_1}}\\
				&+\sum_{j=1}^{k} \frac{1}{(\lam |x_j^--x_1^+|)^{n-4s-\theta_1}}\Big)\Big)+O\Big(\frac{1}{\lam ^{2s+1+\var }}\Big)\Big)\\
				=&2k\Big(-\frac{sB_1 }{\lam ^{2s+1}} + O\Big(\frac{1}{\lam ^{2s+1+\var }}\Big)\Big),
			\end{align*}
			where $B_1=V(r_0,y_0'')\int U^2_{0,1}>0$, $\theta_1>0$ is a small constant and we also used the fact $n-4s-\theta_1>1$ and $2s+\beta+\frac{2s(n-4s-\theta_1)}{n-2s}>2s+1+\var$ in the last step.
			
			Using Lemma \ref{lemA.b2} again, we have
			\begin{align*}
				I_2=&2 k\Big(\int(1-K(r, y'')) U_{x_1^{+}, \lam}^{2_s^*-1} \frac{\pa U_{x_1^{+}, \lam}}{\pa \lam}
				+O\Big(\frac{1}{\lam^{\beta}} \int(1-K(r, y'')) U_{x_1^{+}, \lam}^{2_s^*-1}\Big(\sum_{j=2}^k U_{x_j^{+}, \lam}+\sum_{j=1}^k U_{x_j^{-}, \lam}\Big)\Big)\Big)\\
				=&2k\Big(\int_{B_{\lam^{-\frac{1}{2}+\var}}(x_1^{+})}(1-K(r, y'')) U_{x_1^{+}, \lam}^{2_s-1} \frac{\pa U_{x_1^{+}, \lam}}{\pa \lam}+O\Big(\frac{1}{\lam^{2 s+1+\var}}\Big)\Big)\\
				=&2kO\Big(\frac{1}{\lam^{2 s+1+\var}}\Big).
			\end{align*}
			Finally, we estimate $I_3$. By symmetry and Lemma \ref{lemA.1}, we have
			\begin{equation*}
				\begin{aligned}
					I_3=&2k\int_{\Omega^+_1}\Big((Z_{\bar{r},\bar{h},\bar{y}'',\lam }^*)^{2_s^*-1}-\sum_{j=1}^{k}U_{x_j^+,\lam }^{2_s^*-1}-\sum_{j=1}^{k}U_{x_j^-,\lam }^{2_s^*-1}\Big)\frac{\pa  Z_{\bar{r},\bar{h},\bar{y}'',\lam }^*}{\pa  \lam }\\
					=&2k\Big(\int_{\Omega^+_1}(2_s^*-1)U_{x_1^+,\lam }^{2_s^*-2}\Big(\sum_{j=2}^{k}U_{x_j^+,\lam }+\sum_{j=1}^{k}U_{x_j^-,\lam }\Big)\frac{\pa  U_{x_1^+,\lam }}{\pa  \lam } +O\Big(\frac{1}{\lam ^{2s+1+\var }}\Big)\Big)\\
					=&2k\Big(-\sum_{j=2}^{k}\frac{B_2}{\lam ^{n-2s+1}|x_j^+-x_1^+|^{n-2s}} - \sum_{j=1}^{k}\frac{B_2}{\lam ^{n-2s+1}|x_j^--x_1^+|^{n-2s}}
					+O\Big(\frac{1}{\lam ^{2s+1+\var }}\Big)\Big),
				\end{aligned}
			\end{equation*}
			for some constant $B_2>0$.
			
			Thus, we obtain that
			\begin{align*}
				\frac{\pa  I(Z_{\bar{r},\bar{h},\bar{y}'',\lam })}{\pa  \lam } =& 2k\Big(-\frac{sB_1}{\lam ^{2s+1}}
				+ \sum_{j=2}^{k}\frac{B_2}{\lam ^{n-2s+1}|x_j^+-x_1^+|^{n-2s}} \\
				&+ \sum_{j=1}^{k}\frac{B_2}{\lam ^{n-2s+1}|x_j^--x_1^+|^{n-2s}}
				+ O\Big(\frac{1}{\lam ^{2s+1+\var }}\Big)\Big).
			\end{align*}
		\end{proof}
		As in \cite{guo2020solutions}, we can prove the following lemma.
		\begin{lem}If $n>4s+1$, then we have
			\begin{align*}
				\frac{\pa I\left(Z_{\bar{r}, \bar{h}, \bar{y}'', \lam  }\right)}{\pa \bar{r}}= & 2 k\Big(\frac{sB'_1}{\lam ^{2s}}+\sum_{j=2}^k \frac{B'_2}{\bar{r} \lam  ^{n-2s}|x_j^{+}-x_1^{+}|^{n-2s}}+\sum_{j=1}^k \frac{B'_2}{\bar{r} \lam  ^{n-2s}|x_j^{-}-x_1^{+}|^{n-2s}}. \\
				& +O\Big(\frac{1}{\lam  ^{s+\var}}\Big)\Big),
			\end{align*}
			and
			\begin{align*}
				\frac{\pa I\left(Z_{\bar{r}, \bar{h}, \bar{y}'', \lam  }\right)}{\pa \bar{y}''_j}= 2 k\Big(\frac{sB'_1}{\lam ^{2s}}+O\Big(\frac{1}{\lam  ^{s+\var}}\Big)\Big),\quad j=4,\ldots,n,
			\end{align*}
			where $B'_1$ and $B'_2$ are some positive constants.
		\end{lem}
		
		\section{ Proof of the main results}\label{sec:4}
		In this section, we will choose suitable $(\bar{r},\bar{h},\bar{y}'',\lam )$ such that $u_k=Z_{\bar{r},\bar{h},\bar{y}'',\lam }+\varphi_{\bar{r},\bar{h},\bar{y}'',\lam }$ is a solution of \eqref{1.1}. Following the idea in\cite{PWY2018}, in order to apply local Pohozaev identities, we quote the extension of $u_k$, that is
		\begin{equation*}
			\tilde{u}_k=\tilde{Z}_{\bar{r},\bar{h},\bar{y}'',\lam}+\tilde{\varphi},
		\end{equation*}
		where $\tilde{Z}_{\bar{r},\bar{h},\bar{y}'',\lam}$ and $\tilde{\varphi}$ are the extensions of $Z_{\bar{r},\bar{h},\bar{y}'',\lam}$ and $\varphi$, respectively. Then we have
		\begin{equation}\label{3.1}
			\left\{
			\begin{aligned}
				\operatorname{div}(t^{1-2s}\nabla\tilde{u}_k)=&0,&&\text{ in } \,\R_+^{n+1},\\
				-\lim_{t \to 0^+}t^{1-2s}\pa_t\tilde{u}_k
				=&\omega_s\Big(-V(r,y'')u_k+K(r,y'')(u_k)_+^{2_s^*-1}\\&+\sum_{l=2}^n c_l\sum_{j=1}^k(Z_{x_j^+,\lam}^{2_s^*-2}Z_{j,l}^+
				+Z_{x_j^-,\lam}^{2_s^*-2}Z_{j,l}^-)\Big)
				&&\text{ on }\, \Rn.
			\end{aligned}
			\right.
		\end{equation}
		Without loss of generality, we may assume $\omega_s=1$.
		
		Following the same arguments in \cite{DHW2023A}, we can obtain 
		\begin{prop}\label{prop3.1}
			Suppose that  $(\bar{r},\bar{h},\bar{y}'',\lam )$ satisfies
			\begin{align}
				&-\int_{\pa''\mathcal{B}_\rho^+} t^{1-2s}\frac{\pa\tilde{u}_k}{\pa \nu}\frac{\pa\tilde{u}_k}{\pa y_i}+\frac{1}{2}\int_{\pa''\mathcal{B}_\rho^+} t^{1-2s}|\nabla\tilde{u}_k|^2\nu_i\notag
				\\=&\int_{B_\rho}\Big(-V(r,y'')u_k+K(r,y'')(u_k)_+^{2_s^*-1}\Big)\frac{\pa u_k}{\pa y_i}, \quad i=4,\ldots,n,\label{3.5}
			\end{align}
			\begin{align}	
				&-\int_{\pa''\mathcal{B}_\rho^+} t^{1-2s}\langle\nabla\tilde{u}_k,Y\rangle\frac{\pa\tilde{u}_k}{\pa \nu}+\frac{1}{2}\int_{\pa''\mathcal{B}_\rho^+}t^{1-2s}|\nabla\tilde{u}_k|^2\langle Y,\nu\rangle +\frac{2s-n}{2}\int_{\pa\mathcal{B}_\rho^+}t^{1-2s}\frac{\pa\tilde{u}_k}{\pa \nu}\tilde{u}_k\notag\\\label{3.6}
				=&\int_{{B_\rho}}\Big(-V(r,y'')u_k+K(r,y'')(u_k)_+^{2_s^*-1}\Big)\langle \nabla u_k,y\rangle,
			\end{align}
			and
			\be \label{3.7}
			\int\Big((-\Delta)^s u_k+V(r,y'')u_k-K(r,y'')(u_k)_+^{2_s^*-1}\Big)\frac{\pa  Z_{\bar{r},\bar{h},\bar{y}'',\lam }}{\pa  \lam }=0,
			\ee
			where $u_k=Z_{\bar{r},\bar{h},\bar{y}'',\lam }+\varphi_{\bar{r},\bar{h},\bar{y}'',\lam }$ and $B_\rho=\{(r,y''):|(r,y'')-(r_0,y''_0)|\leq \rho\}$ with $\rho \in (2\delta,5\delta)$. Then $c_l=0$, $l=2,\ldots,n$.
		\end{prop}
		\begin{lem}\label{lemB.2}
			We have 
			\begin{align}
				&\int\Big((-\Delta)^s u_k+V(r,y'')u_k-K(r,y'')(u_k)_+^{2_s^*-1}\Big)\frac{\pa  Z_{\bar{r},\bar{h},\bar{y}'',\lam }}{\pa  \lam }\notag \\
				=&2k\Big(-\frac{sB_1}{\lam ^{2s+1}}+\sum_{j=2}^{k}\frac{B_2}{\lam ^{n-2s+1}|x_j^+-x_1^+|^{n-2s}}
				+\sum_{j=1}^{k}\frac{B_2}{\lam ^{n-2s+1}|x_j^--x_1^+|^{n-2s}}+O\Big(\frac{1}{\lam ^{2s+1+\var }}\Big)\Big)\notag \\ \label{lemB.2.1}
				=&2k \Big(-\frac{sB_1}{\lam ^{2s+1}}+\frac{B_3k^{n-2s}}{\lam ^{n-2s+1}(\sqrt{1-\bar{h} ^2})^{n-2s}}+\frac{B_4k}{\lam ^{n-2s+1}\bar{h}^{n-2s-1}\sqrt{1-\bar{h} ^2}}+O\Big(\frac{1}{\lam ^{2s+1+\var }}\Big)\Big),
			\end{align}
			where $B_i>0$, $i=1,2,3,4$.
		\end{lem}
		\begin{proof}
			By symmetry, we have
			\begin{equation*}
				\begin{aligned}
					&\int\big((-\Delta)^{s}u_k+V(r,y'')u_k-K(r,y'')(u_k)_+^{2_s^*-1}\big)\frac{\pa  Z_{\bar{r},\bar{h},\bar{y}'',\lam }} {\pa \lam }\\
					=&\Big\langle I'(Z_{\bar{r},\bar{h},\bar{y}'',\lam }),\frac{\pa  Z_{\bar{r},\bar{h},\bar{y}'',\lam }} {\pa \lam }\Big\rangle + 2k\int\Big((-\Delta)^s\varphi+V(r,y'')\varphi-(2_s^*-1)K(r,y'')Z_{\bar{r},\bar{h},\bar{y}'',\lam }^{2_s^*-2}\varphi\Big)\frac{\pa  Z_{x_1^+,\lam }} {\pa \lam }\\
					&-\int K(r,y'')((Z_{\bar{r},\bar{h},\bar{y}'',\lam }+\varphi)_+^{2_s^*-1}-Z_{\bar{r},\bar{h},\bar{y}'',\lam }^{2_s^*-1}-(2_s^*-1)Z_{\bar{r},\bar{h},\bar{y}'',\lam }^{2_s^*-2}\varphi)\frac{\pa  Z_{\bar{r},\bar{h},\bar{y}'',\lam }} {\pa \lam }\\
					=&\Big\langle I'(Z_{\bar{r},\bar{h},\bar{y}'',\lam }),\frac{\pa  Z_{\bar{r},\bar{h},\bar{y}'',\lam }} {\pa \lam }\Big\rangle +2kJ_1-J_2.
				\end{aligned}
			\end{equation*}
			Using \eqref{2.13} and \eqref{2.28}, we have
			\begin{equation*}
				\begin{aligned}
					|J_1|=O\Big(\frac{\|\varphi\|_*}{\lam ^{\frac{2s+1+\beta}{2}}}\Big)
					=O\Big(\frac{1}{\lam ^{2s+1+\var }}\Big).
				\end{aligned}
			\end{equation*}
			Note that
			\begin{equation*}
				|(1+t)_+^p-1-pt|\leq
				\left\{
				\begin{array}{ll}
					Ct^2, & \hbox{$1<p\leq2$;} \vspace{0.15cm}\\
					C(t^2+|t|^p), & \hbox{$p>2$.}
				\end{array}
				\right.
			\end{equation*}
			If $2_s^*\leq3$, then we have
			
			\begin{align*}
				|J_2|=&\Big|\int\big((Z_{\bar{r},\bar{h},\bar{y}'',\lam }+\varphi)_+^{2_s^*-1}-Z_{\bar{r},\bar{h},\bar{y}'',\lam }^{2_s^*-1}-(2_s^*-1)Z_{\bar{r},\bar{h},\bar{y}'',\lam }^{2_s^*-2}\varphi\big)\frac{\pa  Z_{x_1^+,\lam }} {\pa \lam }\Big|\\
				\lesssim&\int\Big|Z_{\bar{r},\bar{h},\bar{y}'',\lam }^{2_s^*-3}\varphi^2\frac{\pa  Z_{\bar{r},\bar{h},\bar{y}'',\lam }} {\pa \lam }\Big|\\
				\lesssim&\frac{\|\varphi\|_*^2}{\lam ^\beta}\int\Big(\sum_{j=1}^{k}\frac{\lam ^{\frac{n-2s}{2}}}{(1+\lam |y-x_j^+|)^{n-2s}}+\sum_{j=1}^{k}\frac{\lam ^{\frac{n-2s}{2}}}{(1+\lam |y-x_j^-|)^{n-2s}}\Big)^{2_s^*-2}\\
				&\times\Big(\sum_{j=1}^{k}\frac{\lam ^{\frac{n-2s}{2}}}{(1+\lam |y-x_j^+|)^{\frac{n-2s}{2}+\tau}}+\sum_{j=1}^{k}\frac{\lam ^{\frac{n-2s}{2}}}{(1+\lam |y-x_j^-|)^{\frac{n-2s}{2}+\tau}}\Big)^2\\
				\lesssim&\frac{\|\varphi\|_*^2}{\lam ^\beta}\int\lam^n\Big(\sum_{j=1}^{k}\frac{1}{(1+\lam |y-x_j^+|)^{4s}}+\sum_{j=1}^{k}\frac{1}{(1+\lam |y-x_j^-|)^{4s}}\Big)\\
				&\times\Big(\sum_{j=1}^{k}\frac{1}{(1+\lam |y-x_j^+|)^{n-2s+2\tau}}+\sum_{j=1}^{k}\frac{1}{(1+\lam |y-x_j^-|)^{n-2s+2\tau}}\Big)\\
				\lesssim&\frac{\|\varphi\|_*^2}{\lam ^\beta}\Big(C+\sum\limits_{j=2}^{k}
				\frac{1}{(\lam|x_1^+-x_j^+|)^{2s+\tau}}+\sum\limits_{j=1}^{k}\frac{1}{(\lam|x_1^+-x_j^-|)^{2s+\tau}}\Big)\\
				\leq&\frac{Ck\|\varphi\|_*^2}{\lam ^\beta}=O\Big(\frac{k}{\lam ^{2s+1+\var }}\Big).
			\end{align*}
			
			Similarly, if $2_s^*>3$, we have,
			\begin{align*}
				|J_2|\lesssim\Big|\int\big(Z_{\bar{r},\bar{h},\bar{y}'',\lam }^{2_s^*-3}\varphi^2+\varphi^{2_s^*-1}\big)\frac{\pa  Z_{\bar{r},\bar{h},\bar{y}'',\lam }} {\pa \lam }\Big|
				=O\Big(\frac{k}{\lam ^{2s+1+\var }}\Big).
			\end{align*}
			Hence,
			\begin{align*}
				\Big\langle I'(Z_{\bar{r},\bar{h},\bar{y}'',\lam }+\varphi),\frac{\pa  Z_{\bar{r},\bar{h},\bar{y}'',\lam }} {\pa \lam }\Big\rangle
				=\Big\langle I'(Z_{\bar{r},\bar{h},\bar{y}'',\lam }),\frac{\pa  Z_{\bar{r},\bar{h},\bar{y}'',\lam }} {\pa \lam }\Big\rangle +O\Big(\frac{k}{\lam ^{2s+1+\var }}\Big),
			\end{align*}
			which, together with Lemma \ref{lemB.1}  yields that \eqref{lemB.2.1}. The proof is completed.
		\end{proof}
		Similarly, we can prove the following:
		\begin{align}\label{r}
			&\Big\langle I'(Z_{\bar{r},\bar{h},\bar{y}'',\lam }+\varphi),\frac{\pa  Z_{\bar{r},\bar{h},\bar{y}'',\lam }} {\pa \bar{r}}\Big\rangle \notag\\=&2k\Big(\frac{sB'_1}{\lam ^{2s}}+\sum_{j=2}^{k}\frac{B'_2}{\bar{r}\lam ^{n-2s}|x_j^+-x_1^+|^{n-2s}}
			+\sum_{j=1}^{k}\frac{B'_2}{\bar{r}\lam ^{n-2s}|x_j^--x_1^+|^{n-2s}}+O\Big(\frac{1}{\lam ^{s+\var }}\Big)\Big)
		\end{align}
		and
		\begin{align}\label{y}
			\Big\langle I'(Z_{\bar{r},\bar{h},\bar{y}'',\lam }+\varphi),\frac{\pa  Z_{\bar{r},\bar{h},\bar{y}'',\lam }} {\pa \bar{y}''_j}\Big\rangle =2k\Big(\frac{sB'_1}{\lam ^{2s}}+O\Big(\frac{1}{\lam ^{s+\var }}\Big)\Big).
		\end{align}
		Next, we will estimate \eqref{3.5} and \eqref{3.6} and find their equivalent identities.
		
		For \eqref{3.5}, integrating by parts we have
		\begin{align*}
			\int_{B_\rho}\Big(-V(r,y'')u_k+K(r,y'')(u_k)_+^{2_s^*-1}\Big)\frac{\pa u_k}{\pa y_i}
			=&-\frac{1}{2}\int_{\pa B_\rho}V(r,y'')u_k^2\nu_i+\frac{1}{2_s^*}\int_{\pa B_\rho}K(r,y'')(u_k)_+^{2_s^*} \nu_i\\& +\frac{1}{2}\int_{B_\rho}\frac{\pa V(r,y'')}{\pa y_i}u_k^2- \frac{1}{2_s^*}\int_{B_\rho}\frac{\pa K(r,y'')}{\pa y_i}(u_k)_+^{2_s^*}.
		\end{align*}
		
		Hence, we have that \eqref{3.5} is equivalent to
		\be\label{3.20}
		\begin{aligned}
			\frac{1}{2}\int_{B_\rho}\frac{\pa V(r,y'')}{\pa y_i}u_k^2-\frac{1}{2_s^*}\int_{B_\rho}\frac{\pa K(r,y'')}{\pa y_i}(u_k)_+^{2_s^*} =&\frac{1}{2}\int_{\pa B_\rho}V(r,y'')u_k^2\nu_i -\frac{1}{2_s^*}\int_{\pa B_\rho}K(r,y'')(u_k)_+^{2_s^*-1} \nu_i\\&-\int_{\pa''\mathcal B_\rho^+} t^{1-2s}\frac{\pa\tilde{u}_k}{\pa \nu}\frac{\pa\tilde{u}_k}{\pa y_i}+\frac{1}{2}\int_{\pa''\mathcal B_\rho^+} t^{1-2s}|\nabla\tilde{u}_k|^2\nu_i,
		\end{aligned}
		\ee
		for $i=4,\ldots,n$.
		
		On the other hand, from \eqref{3.1}, we can obtain
		\begin{equation*}
			\begin{aligned}
				\int_{\pa\mathcal B_\rho^+} t^{1-2s}\frac{\pa\tilde{u}_k}{\pa \nu}\tilde{u}_k  =&\int_{\pa''\mathcal B_\rho^+}t^{1-2s}\frac{\pa\tilde{u}_k}{\pa \nu}\tilde{u}_k \\
				&+ \int_{B_\rho}\Big(-V(r,y'')u_k+K(r,y'')(u_k)_+^{2_s^*-1}+\sum_{l=2}^n
				c_l\sum_{j=1}^k(Z_{x_j^+,\lam}^{2_s^*-2}Z_{j,l}^+ +Z_{x_j^-,\lam}^{2_s^*-2}Z_{j,l}^-)\Big)u_k.
			\end{aligned}
		\end{equation*}
		A direct computation gives
		\begin{equation*}
			\begin{aligned}
				\int_{B_\rho}V(r,y'')u_k\langle \nabla u_k, y\rangle
				=&\frac{1}{2}\int_{\pa B_\rho}V(r,y'')u_k^2\langle y,\nu\rangle \\&- \frac{1}{2}\int_{B_\rho}\langle y,\nabla V(r,y'')\rangle u_k^2 - \frac{n}{2}\int_{B_\rho}V(r,y'')u_k^2,
			\end{aligned}
		\end{equation*}
		and
		\begin{equation*}
			\begin{aligned}
				\int_{B_\rho}K(r,y'')(u_k)_+^{2_s^*-1}\langle \nabla u_k, y\rangle
				=&\frac{1}{2_s^*}\int_{\pa B_\rho}K(r,y'')(u_k)_+^{2_s^*} \langle y,\nu\rangle \\&- \frac{1}{2_s^*}\int_{B_\rho}\langle y,\nabla K(r,y'')\rangle (u_k)_+^{2_s^*} - \frac{n}{2_s^*}\int_{B_\rho}K(r,y'')(u_k)_+^{2_s^*},
			\end{aligned}
		\end{equation*}
		which, together with
		\begin{equation*}
			\sum_{l=2}^n c_l \sum_{j=1}^k\int (Z_{x_j^+,\lam}^{2_s^*-2}Z_{j,l}^+ + Z_{x_j^-,\lam}^{2_s^*-2}Z_{j,l}^-)\varphi=0,
		\end{equation*}
		yields that \eqref{3.6} is equivalent to
		\begin{align}
			&\int_{B_\rho}\Big(sV(r,y'')+\frac{1}{2}\langle y,\nabla V(r,y'')\rangle\Big) u_k^2 -\frac{1}{2_s^*}\int_{B_\rho}\langle \nabla K(r,y''), y\rangle(u_k)_+^{2_s^*} \notag\\=&-\int_{\pa''\mathcal B_\rho^+}t^{1-2s}\langle\nabla\tilde{u}_k,Y\rangle\frac{\pa\tilde{u}_k}{\pa \nu}+\frac{1}{2}\int_{\pa''{\mathcal B_\rho^+}}t^{1-2s}|\nabla\tilde{u}_k|^2\langle Y,\nu\rangle 
			+\frac{2s-n}{2}\int_{\pa''{\mathcal B_\rho^+}}t^{1-2s}\frac{\pa\tilde{u}_k}{\pa \nu}\tilde{u}_k\notag\\&
			-\frac{1}{2_s^*}\int_{\pa B_\rho}K(r,y'')(u_k)_+^{2_s^*} \langle y,\nu\rangle+\frac{1}{2}\int_{\pa B_\rho}V(r,y'')u_k^2\langle y,\nu\rangle\notag\\
			&+ \frac{2s-n}{2}\sum_{l=2}^n c_l \int_{B_\rho}\Big(\sum\limits_{j=1}^k\big(Z_{x_j^+,\lam}^{2_s^*-2}Z_{j,l}^+ +Z_{x_j^-,\lam}^{2_s^*-2}Z_{j,l}^-\big)\Big)Z_{\bar{r},\bar{h},\bar{y}'',\lam}.\label{3.19}
		\end{align}
		\begin{lem}\label{lem3.2}
			Relations \eqref{3.20} and \eqref{3.19} are, respectively, equivalent to
			\begin{equation}\label{3.22}
				\frac{1}{2}\int_{B_\rho}\frac{\pa V(r,y'')}{\pa y_i}u_k^2-\frac{1}{2_s^*}\int_{B_\rho}\frac{\pa K(r,y'')}{\pa y_i}(u_k)_+^{2_s^*}= O\Big(\frac{k}{\lam^{2s+\var}}\Big), \quad i=4,\ldots,n.
			\end{equation}
			and
			\begin{equation}\label{3.21}
				\int_{B_\rho}(sV(r,y'')+\frac{1}{2}\langle y,\nabla V(r,y'')\rangle) u_k^2-\frac{1}{2_s^*}\int_{B_\rho}\langle \nabla K(r,y''), y\rangle(u_k)_+^{2_s^*} = O\Big(\frac{k}{\lam^{2s+\var}}\Big)
			\end{equation}
		\end{lem}
		
		\begin{proof}
			Since the proof is similar to Lemma 3.2 in \cite{DHW2023A}, we omit it here.
		\end{proof}
		Next, we prove
		\begin{lem}\label{lem3.3}
			For any function $g(r,y'') \in C^1(\Rn)$, there holds
			$$
			\int_{B_\rho}g(r,y'')u_k^{2_s^*} = 2k\Big(g(\bar{r},\bar{y}'')\int U_{0,1}^{2_s^*}+o(1)\Big),
			$$
			and 
			$$\int_{B_\rho}g(r,y'')u_k^2 = 2k\Big(\frac{1}{\lam ^{2s}}g(\bar{r},\bar{y}'')\int U_{0,1}^2+o\Big(\frac{1}{\lam ^{2s}}\Big)\Big).$$
		\end{lem}
		
		\begin{proof}
			Since the proof is similar to Lemma 3.3 in \cite{L2024}, we omit it here.
		\end{proof}
		
		\subsection{Proof of Theorem \ref{theo1.1}}
		Now we will complete the proof of Theorem \ref{theo1.1}.
		
		Apply Lemma \ref{lem3.3} to \eqref{3.21} and \eqref{3.22}, we obtain the equations to determine $(\bar{r}, \bar{y}'')$ are
		\be \label{3.37}
		\frac{\pa  K(\bar{r}, \bar{y}'')}{\pa  \bar{r}}=o(1),
		\ee
		and
		\be \label{3.38}
		\frac{\pa  K(\bar{r}, \bar{y}'')}{\pa  \bar{y}_i''}=o(1), \quad i=4, \ldots, n .
		\ee
		From \eqref{3.7} and Lemma \ref{lemB.2}, the equation to determine $\lam$ is
		\be \label{3.39}
		-\frac{sB_1}{\lam ^{2s+1}}+\frac{B_3k^{n-2s}}{\lam ^{n-2s+1}(\sqrt{1-\bar{h} ^2})^{n-2s}}+\frac{B_4k}{\lam ^{n-2s+1}\bar{h}^{n-2s-1}\sqrt{1-\bar{h} ^2}}=O\Big(\frac{1}{\lam ^{2s+1+\var }}\Big).
		\ee
		
		\begin{proof}[Proof of Theorem \ref{theo1.1}]Through the discussion above, we have proved that \eqref{3.5}, \eqref{3.6} and \eqref{3.7} are equivalent to \eqref{3.37}, \eqref{3.38} and \eqref{3.39}.

			Letting $\lam =tk^{\frac{n-2s}{n-4s-\al }}$ with $\al =n-4s-\nu$,
			then $t\in[L_0,L_1]$. From \eqref{3.39}, we have for each $\bar{h}$ satisfying $\sqrt{1-\bar{h} ^2}=\frac{1}{M_1\lam ^{\frac{\al }{n-2s}}}$,
			\begin{equation*}
				-\frac{sB_1}{t^{2s+1}}+\frac{B_3M_1^{n-2s}}{t^{n-2s+1-\al }}=o(1), \quad t\in[L_0,L_1].
			\end{equation*}
			
			Define
			\begin{equation*}
				F(t,\bar{r},\bar{y}'')=\Big(\nabla_{\bar{r},\bar{y}''}K(\bar{r},\bar{y}''),-\frac{sB_1}{t^{2s+1}}+\frac{B_3M_1^{n-2s}}{t^{n-2s+1-\al }}\Big).
			\end{equation*}
			Then, it holds
			\begin{equation*}
				\deg\Big(F(t,\bar{r},\bar{y}''),[L_0,L_1]\times B_{\theta_k}(r_0,y''_0)\Big)
				=-\deg\Big(\nabla_{\bar{r},\bar{y}''}K(\bar{r},\bar{y}''),B_{\theta_k}(r_0,y''_0)\Big)\neq0.
			\end{equation*}
			Hence, \eqref{3.37}, \eqref{3.38} and \eqref{3.39} have a solution $t_k\in[L_0,L_1]$, $(\bar{r}_k,\bar{y}''_k)\in B_{\theta_k}(r_0,y''_0)$.
		\end{proof}
		
		\subsection{Proof of Theorem \ref{theo1.4}}
		In this part, we give a brief proof of Theorem \ref{theo1.4}. We assume that $k>0$ is a large integer, $\lam  \in [L'_0k^{\frac{n-2s}{n-4s}},L'_1k^{\frac{n-2s}{n-4s}}]$ for some constant $L'_1>L'_0>0$ and $(\bar{r},\bar{h},\bar{y}'')$ satisfies \eqref{1.7}.
		We define $\tau=\frac{n-4s}{2(n-2s)}$.

		Actually the proof of Theorem \ref{theo1.4} has the same reduction structure as that of Theorem \ref{theo1.1}
		in section \ref{sec:2}. The main difference between the two proofs is how to deal with some problems arising from the distance between these points $\{x_j^+\}_{j=1}^k$ and $\{x_j^-\}_{j=1}^k$. Since the distance between $x_i^+$ and $x_j^+$, for $i\neq j$, will be closer in case (i) and the distance between $x_i^+$ and $x_i^-$ will be closer in case (iii). Therefore, in order to avoid some tedious steps, we just point out that some important estimates, about dealing with the distance between these points $\{x_j^+\}_{j=1}^k$ and $\{x_j^-\}_{j=1}^k$, are still valid in this section under the assumption \eqref{1.7}.
		\begin{proof}[Proof of Theorem \ref{theo1.4}.] We can verify that \eqref{a2.14} and \eqref{a2.14.1} are still valid provided that $\bar{h}$ satisfies \eqref{1.7} and $\frac{n-4s}{n-2s}\leq\bar{\gamma}<1$. Indeed, It follows from Lemma \ref{lemA.4} and \eqref{1.7}
			\begin{equation*}
				\sum_{j=2}^k\frac{1}{(\lam |x_1^+-x_j^+|)^{\bar{\gamma}}}
				\leq \frac{Ck}{(\lam \sqrt{1-\bar{h}^2})^{\bar{\gamma}}}
				\leq \frac{Ck}{\lam ^{\bar{\gamma}}}\leq C,
			\end{equation*}
			\begin{equation*}
				\begin{aligned}
					\sum_{j=1}^k \frac{1}{(\lam |x_1^+-x_j^-|)^{\bar{\gamma}}}
					&\leq\sum_{j=2}^k\frac{1}{(\lam |x_1^+-x_j^+|)^{\bar{\gamma}}}+\frac{C}{(\lam \bar{r}\bar{h})^{\bar{\gamma}}}\leq C.
				\end{aligned}
			\end{equation*}
			
			Similarly, \eqref{b2.14} and \eqref{b2.14.1} are valid for $\bar{h}$ satisfying \eqref{1.7} and $\bar{\gamma}>1$, that is,
			\begin{equation*}
				\sum_{j=2}^k\frac{1}{(\lam |x_1^+-x_j^+|)^{\bar{\gamma}}}
				\leq \frac{Ck^{\bar{\gamma}}}{(\lam \sqrt{1-\bar{h}^2})^{\bar{\gamma}}}
				\leq \frac{Ck^{\bar{\gamma}}}{\lam ^{\bar{\gamma}}}=O\Big(\frac{1}{\lam ^\frac{2s\bar{\gamma}}{n-2s}}\Big),
			\end{equation*}
			and
			\begin{equation*}
				\sum_{j=1}^k \frac{1}{(\lam |x_1^+-x_j^-|)^{\bar{\gamma}}}
				\leq \frac{Ck^{\bar{\gamma}}}{(\lam \sqrt{1-\bar{h}^2})^{\bar{\gamma}}}+\frac{C}{(\lam \bar{r}\bar{h})^{\bar{\gamma}}}=O\Big(\frac{1}{\lam ^\frac{2s\bar{\gamma}}{n-2s}}\Big).
			\end{equation*}
			
			If $\bar{h}$ satisfying \eqref{1.7}, together with Lemma \ref{lemA.b2}, then we have
			\begin{equation*}
				|Z_{j,2}^\pm|\leq C\lam ^{-\beta'}Z_{x_j^{\pm},\lam },
			\end{equation*}
			where $\beta'=\frac{n-4s}{n-2s}$.
			
			By the similar steps in the proof of Theorem \ref{theo1.1}, we can check that \eqref{3.5}, \eqref{3.6} and \eqref{3.7} are equivalent to
			\be \label{4.3}
			\frac{\pa  K(\bar{r}, \bar{y}'')}{\pa  \bar{r}}=o(1),
			\ee
			\be \label{4.4}
			\frac{\pa  K(\bar{r}, \bar{y}'')}{\pa  \bar{y}_i''}=o(1), \quad i=4, \ldots, n .
			\ee
			and
			\be \label{4.5}
			-\frac{sB_1}{\lam ^{2s+1}}+\frac{B_3k^{n-2s}}{\lam ^{n-2s+1}(\sqrt{1-\bar{h} ^2})^{n-2s}}+\frac{B_4k}{\lam ^{n-2s+1}\bar{h}^{n-2s-1}\sqrt{1-\bar{h} ^2}}=O\Big(\frac{1}{\lam ^{2s+1+\var }}\Big).
			\ee
			
			Let $\lam =tk^{\frac{n-2s}{n-4s}}$, then $t\in[L'_0,L'_1]$.
			
			Next, we discuss the main items in \eqref{4.5}.

			Case I: If $(\lam ^{\frac{n-4s}{n-2s}}\bar{h})^{-1}=o(1)$ and $\bar{h}=a\in(0,1)$ as $\lam  \to \infty$, then we get
			\be \label{4.6}
			-\frac{sB_1}{t^{2s+1}}+\frac{B_3}{t^{n-2s+1}(\sqrt{1-a^2})^{n-2s}}=o(1), \quad t\in[L'_0,L'_1].
			\ee
			Define
			\begin{equation*}
				F(t,\bar{r},\bar{y}'')=\Big(\nabla_{\bar{r},\bar{y}''}K(\bar{r},\bar{y}''),-\frac{sB_1}{t^{2s+1}}+\frac{B'_3}{t^{n-2s+1}}\Big).
			\end{equation*}
			Then, it holds
			\be \label{4.7}
			\deg\Big(F(t,\bar{r},\bar{y}''),[L_0,L_1]\times B_{\bar{\theta}_k}(r_0,y''_0)\Big)
			=-\deg\Big(\nabla_{\bar{r},\bar{y}''}K(\bar{r},\bar{y}''),B_{\bar{\theta}_k}(r_0,y''_0)\Big)\neq0.
			\ee
			Hence, \eqref{4.3}, \eqref{4.4} and \eqref{4.5} have a solution $t_k\in[L'_0,L'_1]$, $(\bar{r}_k,\bar{y}''_k)\in B_{\bar{\theta}_k}(r_0,y''_0)$.
			
			Case II: If $(\lam ^{\frac{n-4s}{n-2s}}\bar{h})^{-1} \in (C_1,M_2)$ for some positive constants $C_1$ and large $\lam $, then $\bar{h}=o(1)$ as $\lam  \to \infty$. In fact, in this case, $B_4$ defined in \eqref{4.5} may not be a constant, but we can take a subsequence $\lam _k$ such that $\lim\limits_{\lam \to\infty}(\lam ^{\frac{n-4s}{n-2s}}\bar{h})^{-1}=A\in (C_1,M_2)$. Thus we have
			\begin{equation*}
				-\frac{sB_1}{t^{2s+1}}+\frac{B_3}{t^{n-2s+1}}+\frac{B_4A^{n-2s-1}}{t^{2s+1+\frac{n-4s}{n-2s}}}=o(1), \quad t\in[L'_0,L'_1].
			\end{equation*}
			
			Since $n-2s+1$ and $2s+1+\frac{n-4s}{n-2s}$ are strictly greater than $2s+1$ for $n> 4s+1$, there exists a solution of \eqref{4.3} to \eqref{4.5} like before. Hence, according to Proposition \ref{prop3.1}, we have proved our desired result.
		\end{proof}
		\begin{proof}[Proof of Corollary \ref{cor1.1}.] Corollary \ref{cor1.1} can be proved similarly to Theorem \ref{theo1.4}, we omit it here.\end{proof}
		\appendix
		
		\section{Basic Estimates}\label{secA}
		In this section, we give some essential estimates that can be found in some appropriate references.
		\begin{lem}(Lemma B.1, \cite{wei2010infinitely})\label{lemA.1}
			For $x_i,x_j,y \in \Rn$, define
			$$
			g_{i,j}(y)=\frac{1}{(1+|y-x_i|)^{\kappa_1}}\frac{1}{(1+|y-x_j|)^{\kappa_2}},
			$$
			where $x_i\neq x_j$, $\kappa_1\geq1$ and $\kappa_2\geq1$ are two constants. Then for any constant $0<\varsigma\leq \min(\kappa_1,\kappa_2)$, there is a constant $C>0$, such that
			\begin{equation*}
				g_{i,j}(y)\leq\frac{C}{|x_i-x_j|^\varsigma}\Big(\frac{1}{(1+|y-x_i|)^{\kappa_1+\kappa_2-\varsigma}}+\frac{1}{(1+|y-x_j|)^{\kappa_1+\kappa_2-\varsigma}}\Big).
			\end{equation*}
		\end{lem}
		\begin{lem}(Lemma 2.1, \cite{guo2020solutions})\label{lemA.2}
			For any constant $0<\theta<n-2s$, there is a constant $C>0$, such that
			$$
			\int_{\Rn}\frac{1}{|y-z|^{n-2s}}\frac{1}{(1+|z|)^{2s+\theta}}\, \d z \leq \frac{C}{(1+|y|)^{\theta}}.
			$$
		\end{lem}
		Recall that for $j=1,\ldots,k$,
		\be \label{A.0}
		x_j^\pm=\Big(\bar{r} \sqrt{1-\bar{h} ^2} \cos \frac{2(j-1)\pi}{k}, \bar{r}\sqrt{1-\bar{h}^2}\sin\frac{2(j-1)\pi}{k},\pm\bar{r}\bar{h},\bar{y}''\Big).
		\ee

		\begin{lem}(Lemma A.3,\cite{DHW2023A})\label{lemA.b2}
			As $\lam  \to \infty$, we have
			\begin{equation*}
				\begin{aligned}
					\frac{\pa  U_{x_j^\pm,\lam }}{\pa  \lam }=O(\lam ^{-1}U_{x_j^\pm,\lam })+O(\lam  U_{x_j^\pm,\lam })\frac{\pa  \sqrt{1-\bar{h}^2}}{\pa  \lam }+O(\lam  U_{x_j^\pm,\lam })\frac{\pa  \bar{h}}{\pa  \lam }.
				\end{aligned}
			\end{equation*}
			In particular, if $\sqrt{1-\bar{h}^2}=C\lam ^{-\beta_1}$ with $0<\beta_1<1$, then
			\begin{equation*}
				\frac{\pa  U_{x_j^\pm,\lam }}{\pa  \lam }=O( \frac{ U_{x_j^\pm,\lam }}{\lam ^{\beta_1}}).
			\end{equation*}
			If $\bar{h}=C\lam ^{-\beta_2}$ with $\beta_2>0$, then 
			\begin{equation*}
				\frac{\pa  U_{x_j^\pm,\lam }}{\pa  \lam }=O( \frac{ U_{x_j^\pm,\lam }}{\lam ^{\min(1,\beta_2)}}).
			\end{equation*}
		\end{lem}
		\begin{lem}\label{lemA.4}
			For any $\gamma>0$, there is a constant $C>0$, such that
			\be \label{A.1}
			\sum_{j=2}^{k}\frac{1}{|x_j^+-x_1^+|^\gamma} \leq \frac{Ck^\gamma}{(\bar{r}\sqrt{1-\bar{h}^2})^\gamma}\sum_{j=2}^{k}\frac{1}{(j-1)^\gamma}
			\leq\left\{
			\begin{array}{ll}
				\frac{Ck^\gamma }{(\bar{r}\sqrt{1-\bar{h}^2})^\gamma}, & \gamma>1; \\
				\frac{Ck^\gamma \text{ln}k}{(\bar{r}\sqrt{1-\bar{h}^2})^\gamma}, & \gamma=1; \\
				\frac{Ck}{(\bar{r}\sqrt{1-\bar{h}^2})^\gamma}, & \gamma<1
			\end{array}
			\right.
			\ee
			and
			\be \label{A.2}
			\sum_{j=1}^{k}\frac{1}{|x_j^--x_1^+|^\gamma} \leq \sum_{j=2}^{k}\frac{1}{|x_j^+-x_1^+|^\gamma}+\frac{C}{(\bar{r}\bar{h})^\gamma}.
			\ee
		\end{lem}
		
		\begin{proof}
			From \eqref{A.0}, we have $|x_j^+-x_1^+|=2\bar{r} \sqrt{1-\bar{h} ^2}\sin\frac{(j-1)\pi}{k}$ for $j=2,\ldots,k$. For any $\gamma>0$, it holds
			\begin{equation*}
				\begin{aligned}
					\sum_{j=2}^{k}\frac{1}{|x_j^+-x_1^+|^\gamma}
					&= \frac{1}{(2\bar{r} \sqrt{1-\bar{h} ^2})^\gamma}\sum_{j=2}^{k}\frac{1}{(\sin\frac{(j-1)\pi}{k})^\gamma}\\
					&\leq \frac{Ck^\gamma}{(\bar{r} \sqrt{1-\bar{h} ^2})^\gamma}\sum_{j=2}^{k}\frac{1}{(j-1)^\gamma}.
				\end{aligned}
			\end{equation*}
			In fact, there exist two constants $c_1, c_2>0$, such that
			\begin{equation*}
				c_1\frac{(j-1)\pi}{k}\leq\sin\frac{(j-1)\pi}{k}\leq c_2\frac{(j-1)\pi}{k}\quad \text{ for } \, j=2,\ldots,\Big[\frac{k}{2}\Big]+1.
			\end{equation*}
			We have proved \eqref{A.1}.
			
			On the other hand, noting that $|x_j^--x_1^+|>|x_j^+-x_1^+|$ for $j=2,\ldots,n$, we have
			\begin{equation*}
				\frac{1}{|x_j^--x_1^+|^\gamma}\leq \frac{1}{|x_j^+-x_1^+|^\gamma}.
			\end{equation*}
			Then \eqref{A.2} follows from \eqref{A.1}.
		\end{proof}

		\begin{lem}(Lemma A.7,\cite{DHW2023A})\label{lemA.3}
			Suppose that $\tau \in (0,\frac{n-2s}{2})$, $y =(y_1,y_2,\ldots,y_n)$. Then there is a small $\theta>0$, such that 
			\begin{equation*}
				\int_{\Rn}\frac{1}{|y-z|^{n-2s}}Z_{\bar{r},\bar{h},\bar{y}'',\lam }^{2_s^*-2}
				\sum_{j=1}^{k}\frac{1}{(1+\lam |z-x_j^\pm|)^{\frac{n-2s}{2}+\tau}}\, \d z \leq C \sum_{j=1}^{k}\frac{1}{(1+\lam |y-x_j^\pm|)^{\frac{n-2s}{2}+\tau+\theta}}.
			\end{equation*}
		\end{lem}
		\begin{lem}(Lemma A.3, \cite{GLN2019})\label{lemC.1}
			Let $\sigma>0$, for any constant $0<\theta<n$, there exists a constant $C>0$, independent of $\sigma$, such that
			\begin{equation*}
				\int_{\Rn\backslash B_\sigma(y)}\frac{1}{|y-z|^{n+2s}}\frac{1}{(1+|z|)^\theta}\,\d z\leq C\Big(\frac{1}{(1+|y|)^{\theta+2s}}+\frac{1}{\sigma^{2s}}\frac{1}{(1+|y|)^\theta}\Big).
			\end{equation*}
		\end{lem}
		\begin{lem}(Lemma A.5, \cite{GLN2019})\label{lemA.7}
			Suppose that $(y-x)^2+t^2=\rho^2$, $t>0$. Then there exists a constant $C>0$ such that
			\begin{equation*}
				|\tilde{Z}_{x_j^\pm,\lam}|\leq\frac{C}{\lam^{\frac{n-2s}{2}}}\frac{1}{(1+|y-x_j^\pm|)^{n-2s}} \quad \text{ and } \quad |\nabla\tilde{Z}_{x_j^\pm,\lam}|\leq\frac{C}{\lam^{\frac{n-2s}{2}}}\frac{1}{(1+|y-x_j^\pm|)^{n-2s+1}},
			\end{equation*}
			where $\tilde{Z}_{x_j^\pm,\lam}$ is the extension of $Z_{x_j^\pm,\lam}$.
		\end{lem}
		\begin{lem}(Lemma A.11,\cite{DHW2023A})\label{lemA.8}
			For any $\delta>0$, there exists $\rho=\rho(\delta)\in(2\delta,5\delta)$ such that when $n>4s+1$,
			\begin{equation*}
				\int_{\pa''\mathcal B_\rho^+}t^{1-2s}|\tilde{\varphi}|^2\,\d S \leq \frac{Ck\|\varphi\|_*^2}{\lam^\tau}\quad \text{ and } \quad
				\int_{\pa''\mathcal B_\rho^+}t^{1-2s}|\nabla\tilde{\varphi}|^2\,\d S \leq \frac{Ck\|\varphi\|_*^2}{\lam^\tau},
			\end{equation*}
			where  $\tilde{\varphi}$ is the extension of $\varphi$ and $C>0$ is a constant, depending on $\delta$.
		\end{lem}
		
		\subsection*{Conflict of interest statement} On behalf of all authors, the corresponding author states that there is no conflict of interest.
		
		\subsection*{Data availability statement} Data sharing not applicable to this article as no datasets were generated or analysed during the current study.

		\medskip
		\noindent T. Li
		
		\noindent School of Mathematical Sciences, Laboratory of Mathematics and Complex Systems, MOE\\
		Beijing Normal University, Beijing 100875, People’s Republic of China\\
		Email: \textsf{ting.li@mail.bnu.edu.cn}
		
		\medskip
		
		\noindent Z. Tang
		
		\noindent School of Mathematical Sciences, Laboratory of Mathematics and Complex Systems, MOE\\
		Beijing Normal University, Beijing 100875, People’s Republic of China\\
		Email: \textsf{tangzw@bnu.edu.cn}
		
		\medskip
		\noindent H. Wang
		
		\noindent School of Mathematical Sciences, Laboratory of Mathematics and Complex Systems, MOE\\
		Beijing Normal University, Beijing 100875, People’s Republic of China\\
		Email: \textsf{hmw@mail.bnu.edu.cn}

		\medskip

		\noindent X. Zhang
		
		\noindent School of Mathematical Sciences, Laboratory of Mathematics and Complex Systems, MOE\\
		Beijing Normal University, Beijing 100875, People’s Republic of China\\
		Email: \textsf{xjzhangk@mail.bnu.edu.cn}

	\end{document}